\documentclass[11pt,oneside]{article}
\usepackage{amssymb,amsmath,amsthm,amscd,epsf,latexsym,verbatim,graphicx,amsfonts}
\input epsf.tex
\usepackage{graphicx}
\usepackage{amssymb}
\usepackage{amsthm}
\usepackage{latexsym}
\usepackage{amsmath}
\usepackage{lscape}
\usepackage{palatino, url, multicol}
\usepackage{subfig}
\usepackage{geometry}                
\usepackage{epstopdf}
\theoremstyle{theorem}
\newtheorem{theorem}{Theorem}[section]

\newtheorem{lemma}[theorem]{Lemma}
\newtheorem{proposition}[theorem]{Proposition}

\theoremstyle{definition}

\theoremstyle{definition}

\theoremstyle{definition}
\newtheorem{remark}[theorem]{Remark}
\theoremstyle{definition}
\newtheorem{example}[theorem]{Example}
\theoremstyle{definition}
\newtheorem{definition}[theorem]{Definition}
\theoremstyle{definition}
\newtheorem{definitions}[theorem]{Definitions}
\theoremstyle{definition}

\title{Cutting sequences, regular polygons, 
\\and the Veech group}
\author{Diana Davis\footnote{Brown University, 151 Thayer Street, Providence, RI 02912; diana@math.brown.edu; Phone: (401) 863-7956; Fax: (401) 863-9013}}

\begin{document}
\maketitle
\vspace{-2.0em}
\begin{abstract}
We describe the cutting sequences associated to geodesic flow on regular polygons, in terms of a combinatorial process called \emph{derivation}. This work is an extension of some of the ideas and results in Smillie and Ulcigrai's recent paper, where the analysis was made for the regular octagon. It turns out that the main structural properties of the octagon generalize in a natural way.
\end{abstract}

\vspace{-2.0em}

\section{Introduction}

In this paper, we will consider a surface obtained by identifying parallel edges of a pair of regular polygons. This creates a translation surface: a surface that is flat everywhere except possibly at a finite number of \emph{cone points}. For example, identifying opposite parallel edges of a regular octagon yields a translation surface that is flat everywhere except at the single vertex, where the cone angle is $8\cdot\frac{3\pi}{4} = 6\pi$. 

Translation surfaces arise in unfolding billiards on rational polygons: Instead of considering a billiard path on the polygon, we use the Zemlyakov-Katok construction to  reflect the polygon across its edges, to form a larger surface on which billiard trajectories follow straight lines, always in the same direction (see \cite{ZK}). When the billiard trajectory hits an edge of the polygon, the polygon is reflected in the edge so that, rather than bouncing off at an angle that equals the angle of incidence, the trajectory continues in a straight line. For example, studying billiards on a rectangular table is equivalent to studying geodesic trajectories on a rectangular torus with edges twice as long, and studying billiards on the ($\pi/5, \pi/5, 3\pi/5$) triangle is equivalent to studying geodesic trajectories on a double regular pentagon translation surface (see \cite{DFT}). 

Surfaces obtained via this unfolding carry a translation structure, and are a special case of translation surfaces:

\begin{definition} \label{translationsurface}
A \emph{translation surface} is a collection of polygons with pairs of parallel edges identified in such a way that the identifications are restrictions of translations, every edge is identified with exactly one other edge, and when two edges are identified, their normal vectors point in opposite directions.
\end{definition}

In addition to being a translation surface, a surface created by identifying edges of regular polygons in this way sometimes has additional affine symmetry, and when its group of linear parts of affine symmetries forms a lattice, it is a Veech surface (for a full definition, see \cite{Veech}). Studying the geodesic flow on Veech surfaces can give us results about billiards on the corresponding polygons (for a definitive survey, see \cite{Masur}).

The Veech group of a Veech surface is its group of linear parts of affine symmetries. The Veech group of the surface formed by identifying opposite parallel edges of an $n$-gon (or a pair of $n$-gons for odd $n$) is the $(n,\infty,\infty)$ triangle group. Triangle groups are the simplest possible examples of Veech groups, and they give rise to lots of interesting mathematics. We use an element of the Veech group to shear our polygons, then cut them up and reassemble them back into their original shape, while respecting the side identifications of the surface. We are interested in geodesic trajectories on these surfaces, their associated cutting sequences, and the effect on the cutting sequences of applying the shearing Veech element to the surface. 

In this paper, we will frequently abuse language and refer to, for example, ``the octagon surface," when we mean the surface created by identifying parallel edges of an octagon, with an embedded graph consisting of the edges that were glued together to create the surface. Most of our analysis concerns these edges, so it is helpful to think of regular polygons and their edges sitting flat in the plane, rather than thinking about the associated surface with its embedded graph.

One can consider linear trajectories on translation surfaces, which are geodesics with respect to the flat metric. The classical example is to draw a line on the square torus, or equivalently to draw a line on a square grid (see Section \ref{square}). As the trajectory flows across the polygon, we write down the letter corresponding to each edge it crosses, which gives us a bi-infinite sequence of letters called a \emph{cutting sequence}. We will consider only bi-infinite trajectories, i.e. those that do not pass through corners. When our polygon has an odd number of edges (which is the case we study in this paper), we use two copies that are reflections of each other, so that we have pairs of parallel edges to identify.

On the square torus, the cutting sequences are called \emph{Sturmian sequences}, which have been studied extensively (see, for example, \cite{Arnoux} or \cite{Lothaire}).  A trajectory on a Veech surface satisfies the \emph{Veech dichotomy}: either it is periodic, in which case it can be seen on the polygons as a finite number of line segments from edge to edge, or it is non-periodic, in which case the trajectory is uniformly distributed on the polygons and in fact is dense (see \cite{Veech}, Theorem 1.4, or \cite{SW}, or \cite{Vorobets}). 

Recently, Smillie and Ulcigrai gave a characterization of geodesic trajectories and their corresponding sequences on the surface formed by identifying opposite parallel edges of a regular octagon (\cite{SU}, \cite{SU2}). In the present paper, we present a simplified understanding of the method. The same methods they used for the octagon can be applied to the double pentagon surface, and in fact to all double odd-gons (double regular $n$-gons with odd $n$). 

On a polygon surface, we label the edges $A,B,C,\ldots$ and then look at infinite geodesic paths on the surface. Each time the path crosses an edge, we write down its corresponding letter, and by doing this we obtain a bi-infinite  \emph{cutting sequence} in the letters $\{$A,B,C,\ldots$\}$. 

The basic question is to characterize which sequences arise as cutting sequences. The key tool towards this characterization is to analyze how affine symmetries act on trajectories and cutting sequences, in particular the shearing Veech element. In this paper, we describe and analyze the action of the Veech element on cutting sequences on double $n$-gons for odd $n$. Since this action is the key step toward the characterization of cutting sequences, this shows that the same methods used by \cite{SU} for the octagon and regular $2n$-gons could be applied, and the rest of the characterization could be performed in a similar way.

Given an existing cutting sequence, we can use the symmetries of the polygon to find others. Rotation and reflection give us ``new" sequences by simply permuting the edge labels. But the key feature that makes the octagon a Veech surface is that it has another symmetry: the \emph{shearing Veech element}, a Veech element whose derivative is a parabolic element of $SL_2 (R)$. This element is an affine automorphism that acts as an isomorphism on the surface. Applying the shear to the surface ``twists" it, and transforms a trajectory on the surface in such a way that its corresponding cutting sequence is a genuinely new cutting sequence.

To explain the ``twist," we must first describe the \emph{cylinder decomposition} of a polygonal surface. We will break up our polygons into parallelogram ``cylinders," each with a pair of horizontal edges. To decompose a polygonal surface into cylinders, we align it so that an edge is horizontal, and then cut along horizontal lines through vertices of the polygons. Each cylinder is made from two triangles or two trapezoids, glued together along an edge. An example is given in Figure \ref{cylinderdecomposition} for the double pentagon surface, which has two cylinders. 

\begin{figure}[!h] 
\centering
\includegraphics[width=200pt]{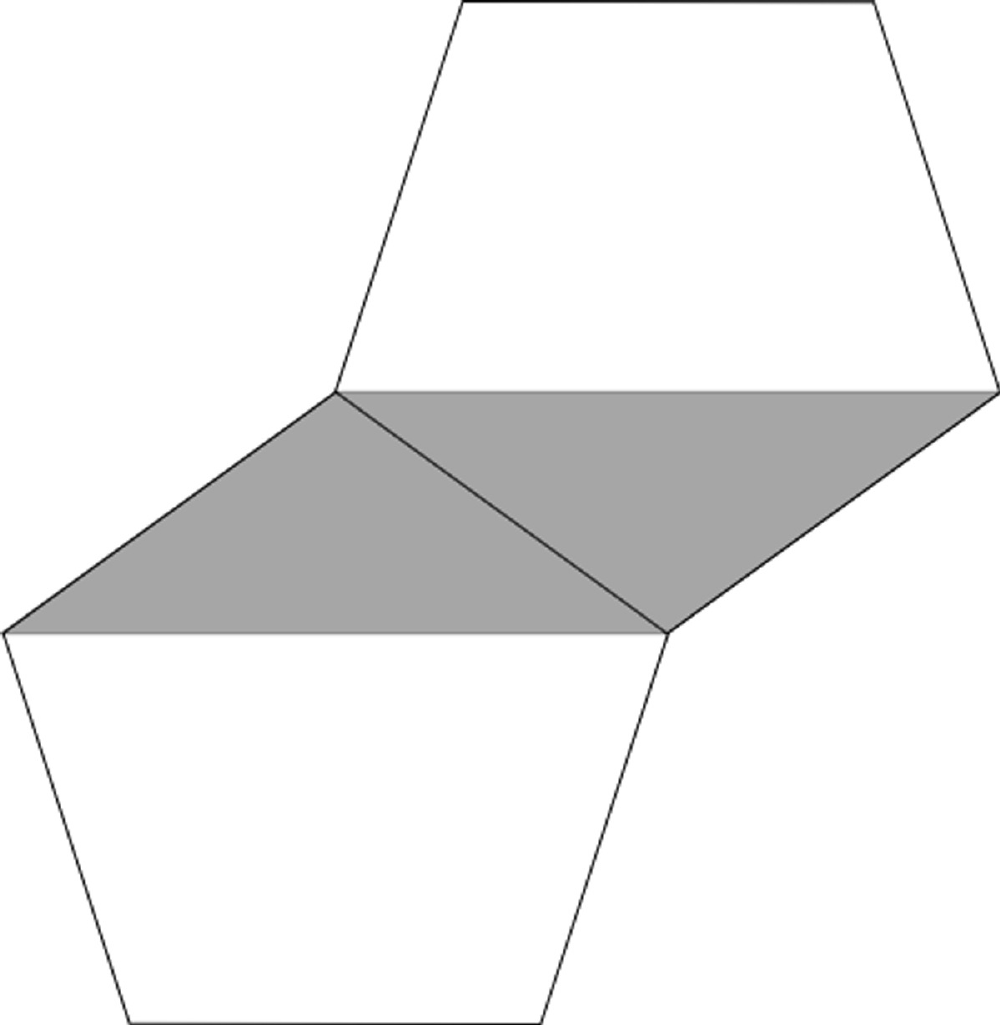}
\begin{quote}\caption{Cylinder decomposition of the double pentagon surface. One cylinder is white; the other is shaded \label{cylinderdecomposition}} \end{quote}
\end{figure} 

The \emph{modulus} of a cylinder is its width divided by its height. In Proposition \ref{veechelt}, we show that all the cylinders of a double odd-gon surface have the same modulus. This modulus appears in the upper-right entry of the $2$x$2$ shearing Veech element matrix, which is the derivative of an affine automorphsm. (The proposition is not a new result; a similar result appears in \cite{Veech}.)

The shearing Veech element acts as a twist in each cylinder, which in the plane looks like pushing a cylinder's top edge to the right an integer number of times. However, it may not twist all the cylinders the same number of times; for example, the Veech shear of the octagon surface shears the rectangular middle cylinder twice and the parallelogram top and bottom cylinder just once (see the octagon on the left side of Figure \ref{octagonshear}). 

We assume that the polygons have unit-length edges, that one edge is horizontal, and that the horizontal edge has endpoints $(0,0)$ and $(1,0)$, where the origin is a fixed point. Then for a single or double polygon, the shearing Veech element, which is the derivative of a particular affine automorphism of the surface, is
\begin{equation*}
M_n = 
\left( \begin{matrix} 1&2\cot(\pi/n)\\ 0&1 \end{matrix} \right).
\end{equation*}
Notice that $M_n$ is a horizontal shear; it only affects the $x$-coordinate of a point, while the $y$-coordinate is unchanged. We will mostly use this orientation-preserving shear because it simplifies the analysis, but the Veech group generator that we will actually apply to the cutting sequences is $\bigl( \begin{smallmatrix} -1&0\\ 0&1 \end{smallmatrix} \bigr) M_n$.

Smillie and Ulcigrai study \emph{derivation} on the octagon, a procedure that takes a trajectory and its corresponding cutting sequence, and shears the octagonal surface using the Veech element $M_8$ to produce a new trajectory and cutting sequence from the original one. By studying the action of the shear on the surface, Smillie and Ulcigrai determined that its effect on sequences is to eliminate all the letters in the cutting sequence except those that are \emph{sandwiched}:

A \emph{sandwiched letter} is a letter that is preceded and followed by the same letter. For example, in the sequence fragment $\ldots ABC\mathbf{D}C\mathbf{CC}C\mathbf{BCB}CDE \ldots $, the sandwiched letters are the ones in bold, so the derived sequence fragment is $\ldots DCCBCB\ldots$. In other words, the derivation rule for the regular octagon is "keep only sandwiched letters."

To formalize this rule, we will define a function $ksl$ which \underline{k}eeps only the \underline{s}andwiched \underline{l}etters in a sequence. For the example above,
\begin{eqnarray*}
s &=& \ldots ABCDCCCCBCBCDE \ldots \\
ksl(s) &=& \ldots DCCBCB \ldots
\end{eqnarray*}

In the work that follows, we refer to periodic sequences by one period, so for example, $BECE$ is shorthand for the periodic sequence $\ldots BECEBECEBECE\ldots$. In that sequence, which we will use extensively in Chapter \ref{deriveBECE}, the sandwiched letters are $B$ and $C$, so the derived sequence is $BC$. In our notation, $ksl(BECE) = BC$.

Note also that if we derive the sequence $BC$, it is unchanged, so we could keep deriving it forever. In contrast, deriving the sequence $ABC$ yields the empty sequence, so it is not derivable. Given a sequence, an important question is whether it could be a cutting sequence. Being derivable is a necessary, but not sufficient, condition for being a cutting sequence. $BECE$ could potentially be (and is) a cutting sequence, while $ABC$ cannot.

\subsection{Results}

$\mathbf{Theorem \ \ref{koslalln}.}$
On a double regular $n$-gon with $n$ odd, the effect of the Veech element $\left ( \begin{smallmatrix} -1&0\\ 0&1 \end{smallmatrix} \right) M_n$ on a cutting sequence $s$ is $ksl(s)$.
\\ \\
In Chapter \ref{deriveBECE}, we show that the derivation rule for a sequence $s$ on the double pentagon is $ksl(s)$. We work through the full argument for the explicit pentagon case, to demonstrate the method for a simple case before generalizing it to the above theorem, which applies to double regular $n$-ons for all odd $n$.
\\ \\
In order to prove this, we need to understand the result of applying a shearing Veech element. In reading Smillie and Ulcigrai's work, we were especially interested in their diagrams of the effect of applying the regular octagon's shearing Veech element $M_8$, as in Figure \ref{octagonshear}. To locate the vertices of the sheared octagon, they used a kind of grid, formed by gluing together multiple copies of a regular octagon.

\begin{figure}[!h] 
\centering
\includegraphics[width=430pt]{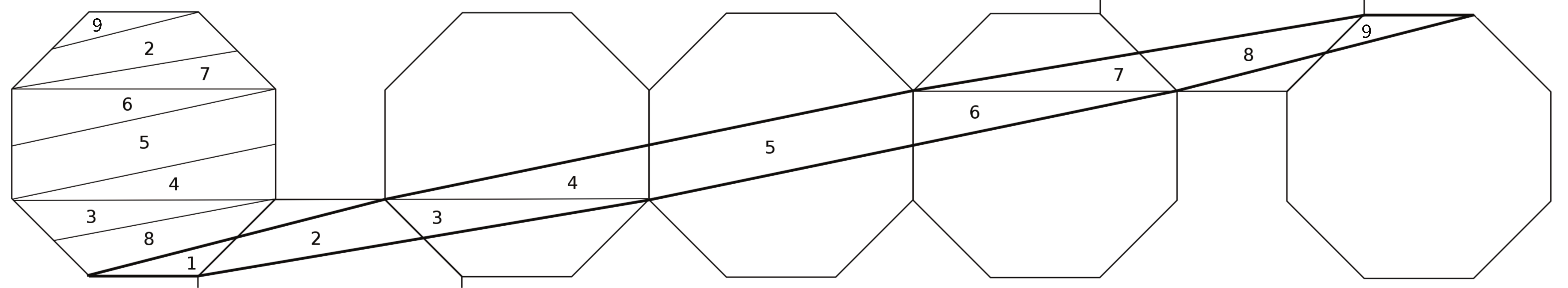}
\begin{quote}\caption[]{The result of applying the matrix $M_8 = \left ( \begin{smallmatrix} 1&2(1+\sqrt{2})\\ 0&1 \end{smallmatrix} \right) $ to the regular octagon \label{octagonshear}} \end{quote}
\vspace{-2.0em}
\end{figure} 

Smillie and Ulcigrai applied the Veech element $M_8$ to the left-most regular octagon in Figure \ref{octagonshear}. This sheared it into the elongated octagon whose parts are consecutively numbered 1 through 9. The sheared octagon can be cut and reassembled back into a regular octagon, as shown. (In the underlying octagonal structure, the second and sixth octagons are not fully included in the figure.)

We call this underlying grid the \emph{vertex generator guide}. We describe how to construct it for the double pentagon in Chapter \ref{vertexguide}, and then in general in Chapter \ref{reassembly}. The following theorem explains the connection between the vertex generator guide and the shearing Veech element:
\\ \\
$\mathbf{Theorem \ \ref{theyrethesame}}$ (Reassembly Theorem).
The effect of applying the shearing Veech element to a double odd-gon is to slide its vertices to corresponding locations on the vertex generator guide.
\\ \\
Several examples of the effect of applying the shearing Veech element to double $n$-gons, including the vertex generator guide and its underlying polygonal structure, are in Figures \ref{pentshear}, \ref{heptagonshear} and \ref{nonagonshear}, for $n=5, 7$ and $9$, respectively. In each figure, the dark sheared polygon is the image of the dark regular polygon, and the light sheared polygon is the image of the light regular polygon. The regular polygons' vertices are sheared to the points on the vertex generator guide, which are highlighted by large dots. These points are vertices on the underlying polygonal structure. 

\begin{figure}[!h]
\centering
\includegraphics[width=350pt]{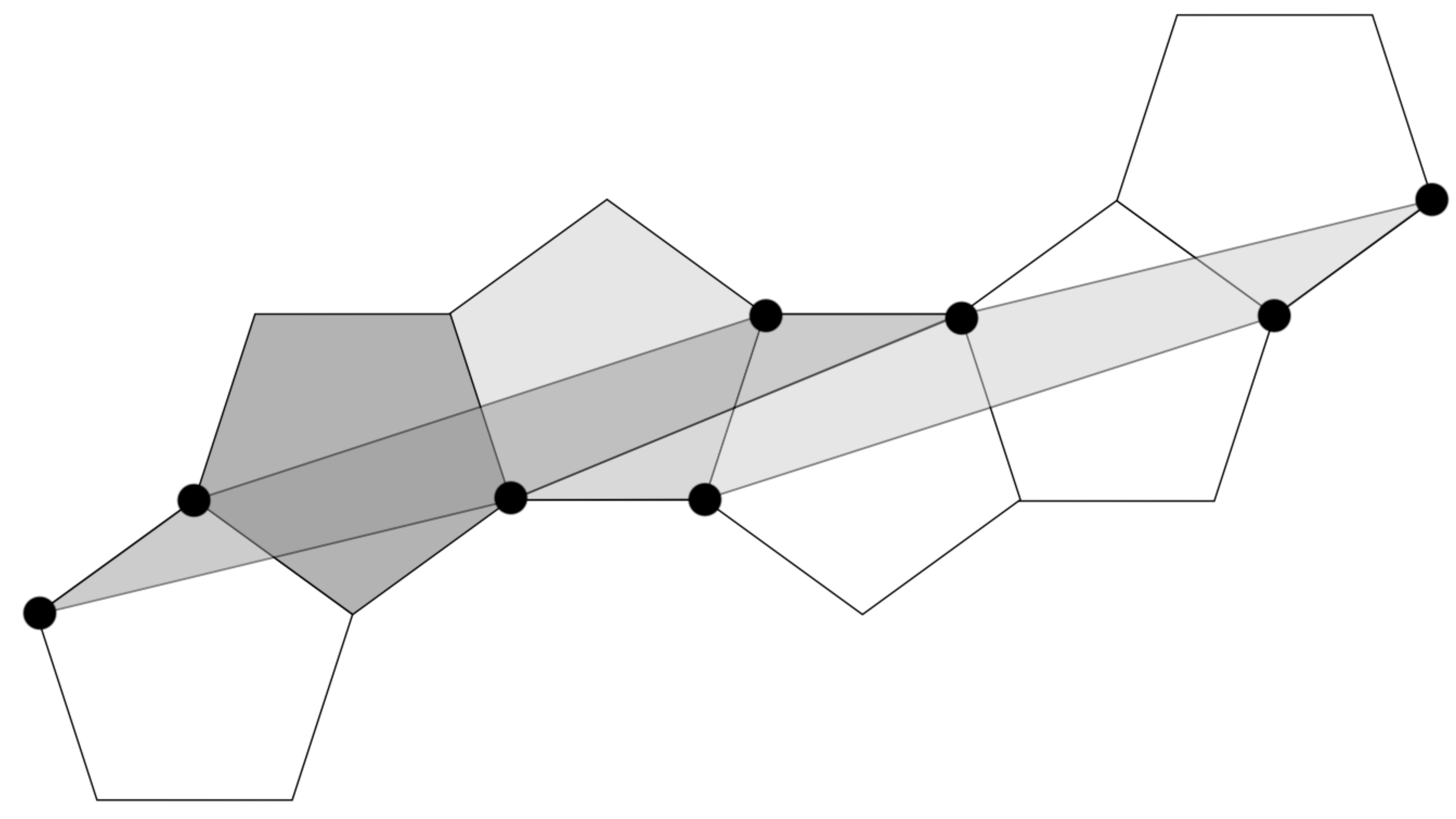}
\begin{quote}\caption{The effect of applying the Veech element to a double pentagon \label{pentshear}} \end{quote}
\vspace{-2.0em}
\end{figure}

\begin{figure}[!h]
\centering
\includegraphics[width=430pt]{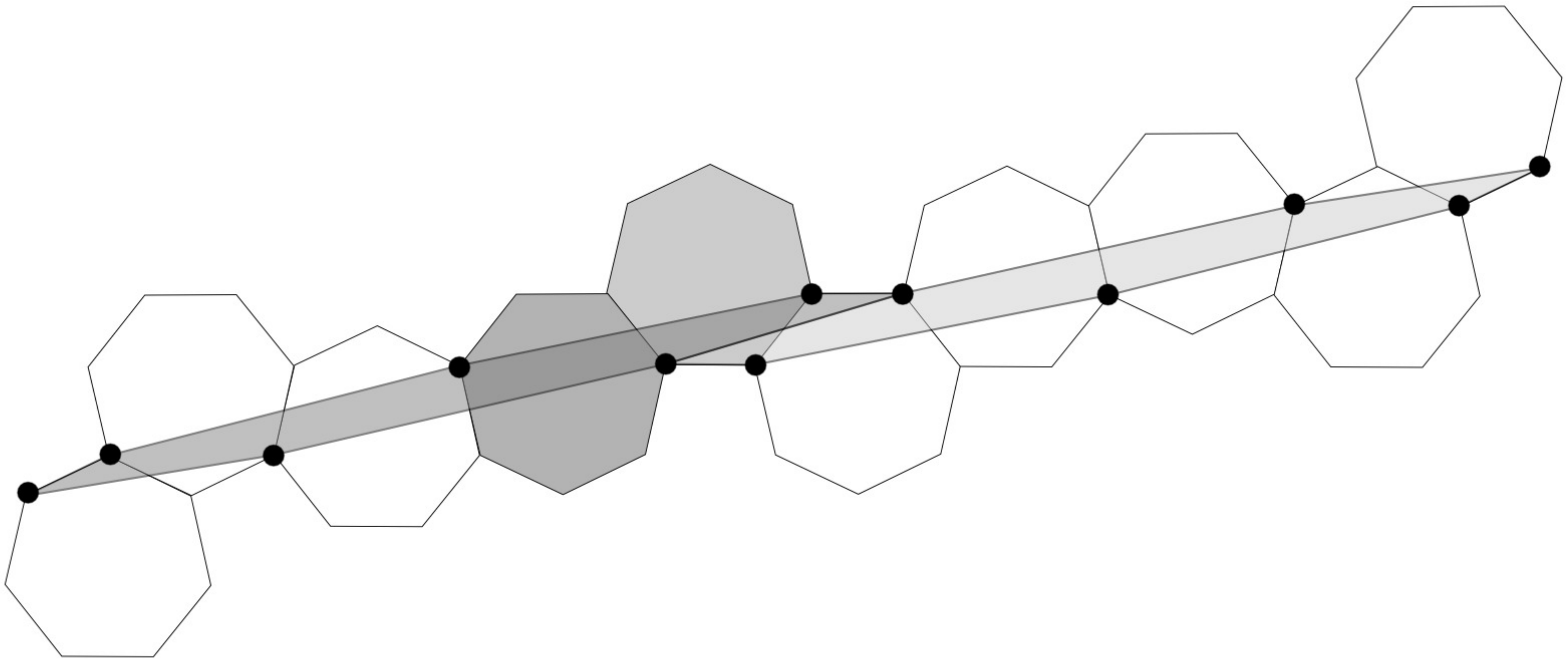}
\begin{quote}\caption{The effect of applying the Veech element to a double heptagon \label{heptagonshear}} \end{quote}
\vspace{-2.0em}
\end{figure}

\begin{figure}[!h]
\centering
\includegraphics[width=430pt]{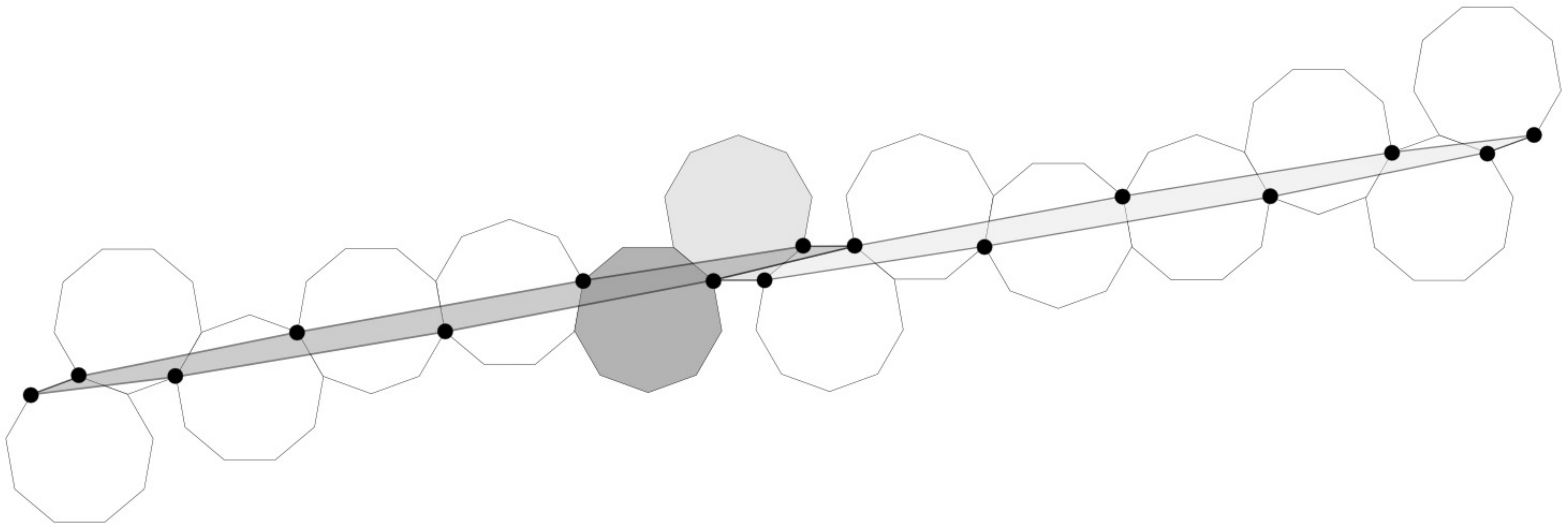}
\begin{quote}\caption{The effect of applying the Veech element to a double nonagon \label{nonagonshear}} \end{quote}
\vspace{-2.0em}
\end{figure}

Once we have the Reassembly Theorem for all odd-gons, we will again use the methods in \cite{SU} for the octagonal case, which we used to prove the derivation rule for the double pentagon (Chapter \ref{deriveBECE}), and apply them to general double odd-gons (Chapter \ref{alln}).

\subsection*{Organization of the paper}
In the second chapter, we give some helpful definitions and lemmas, and explain the geometric process of derivation using the simple example of the square torus. In the third chapter, we define and explain the vertex generator guide, particularly for double odd-gons, which is the basis of the Reassembly Theorem. In the fourth chapter, we explain the geometric process of deriving a sequence on the double pentagon, and then prove that the combinatorial effect on a cutting sequence $s$ is $ksl(s)$. In the fifth chapter, we prove the Reassembly Theorem. In the sixth chapter, we prove that the combinatorial effect of derivation is $ksl(s)$ for a sequence $s$ on any odd-gon. 

More organizational comments will be given at the beginning of each chapter.

\section{Basic properties}

In the first two sections, we give some identities and lemmas that we will use later in the paper. In the third section, we work out the example of the square torus, to illustrate the ideas and methods we will be using throughout the paper.

\subsection{Some Trigonometric Identities}

\begin{lemma}  \label{identities}
The following trigonometric identity holds:
\[
\cot(\theta/2)\sin(k\theta) = 1 + 2\cos(\theta) + \cdots + 2\cos((k-1)\theta) + \cos(k\theta).
\]
\end{lemma}

\begin{proof} 
To simplify the algebra, we will let $a=2\theta$, so that the desired relation is
\begin{eqnarray*}
\cot a\sin 2ka &=& 1 + 2\cos 2a  + 2\cos 4a +\cdots + 2\cos(2(k-1)a) + \cos 2ka. 
\end{eqnarray*}
We will use two standard trigonometric identities:
\begin{eqnarray}
\sin(u)\cos(v)-\cos(u)\sin(v) &=& \sin(u-v), \ \text{and} \label{additionformula} \\
\sin(u+v)-\sin(u-v) &=& 2\cos(u)\sin(v). \label{productformula}
\end{eqnarray}
First, we will expand $\cot a$ into $\cos a/\sin a$, then add and subtract $\cos 2ka$:
\begin{eqnarray*}
\cot a \sin 2ka  &=& \frac{1}{\sin a}(\cos a\sin 2ka) \\
&=& \frac{1}{\sin a }\left(\cos a \sin 2ka  - \sin a \cos 2ka \right)+\cos 2ka 
\end{eqnarray*}
Now apply identity (\ref{additionformula}):
\begin{eqnarray*}
&=& \frac{1}{\sin a }\sin((2k-1)a)+\cos 2ka
\end{eqnarray*}
Now we trivially add and subtract extra terms in a telescoping series:
\begin{eqnarray*}
= \frac{1}{\sin a }\big(\sin a+(\sin 3a-\sin a)&+&(\sin 5a-\sin 3a)+\ldots \\
&+&(\sin(2k-1)a-\sin(2k-3)a)\big)+\cos 2ka
\end{eqnarray*}
and now we apply identity (\ref{productformula})
\begin{eqnarray*}
&=& \frac{1}{\sin a }(\sin a + 2\sin a\cos 2a + 2\sin a\cos 4a + \ldots + 2\sin a \cos (2(k-1)a))+\cos 2ka \\
&=& 1+2\cos 2a +2\cos 4a + \ldots +  2\cos (2(k-1)a)+\cos 2ka,
\end{eqnarray*}
completing the proof.
\end{proof}

\begin{lemma} \label{identitysum}
\begin{eqnarray*}
\sum_{i=1}^{k} \cot \alpha/2 \sin i\alpha = k+ \sum_{i=1}^k (2(k-i)+1)\cos i\alpha.
\end{eqnarray*}
\end{lemma}

\begin{proof}
We can obtain this identity from repeated application of Lemma \ref{identities}.
\end{proof}

\subsection{The Veech Group}

\begin{definition} \label{veechgroup}
Given a surface $S$, let $A(S)$ be the group of affine diffeomorphisms from $S$ to itself, taking cone points to cone points. Let $f:A(S) \mapsto GL_2 (\mathbb{R})$ be the function that takes a locally affine diffeomorphism to its derivative. The Veech group is the image of $A(S)$ under $f$, which lies in $SL_2 (\mathbb{R})$. For a longer definition, see (\cite{SU}, $\mathsection$ 3.1). Note that the classical definition of the Veech group allows only orientation-preserving diffeomorphisms, but (following \cite{SU}) we will allow orientation-reversing diffeomorphisms, because the Veech shear we will use reverses orientation.
\end{definition}

\begin{proposition} \label{veechelt}
Given two regular $n$-gons, $n$ odd, one of which is a half-turn rotation of the other, identify pairs of parallel sides to create a surface. Then the cylinders in the polygons' cylinder decomposition have the same modulus. Thus there exists an affine automorphism whose derivative is $M_n$ = $\bigl( \begin{smallmatrix} -1&-2\cot\pi/n\\ 0&1 \end{smallmatrix} \bigr)$, so $M_n$ belongs to the Veech group.
\end{proposition}

\begin{remark}
A similar result was originally proved in \cite{Veech} (Theorem 1.1).
\end{remark}

\begin{proof}
We decompose the double $n$-gon into horizontal cylinders. Then the matrix $M_n$ must be a horizontal shear, whose effect is to perform a (possibly multiple) twist on each cylinder. The magnitude of the upper-right entry must then be an integer multiple of the modulus (the ratio of width to height) of each cylinder. 

The width of the $k$th cylinder, counting from the bottom of the $n$-gon that has a horizontal edge at its base, is 
\[
2(1 + 2\cos 2\pi/n  + 2\cos 4\pi/n +\cdots + 2\cos(2(k-1)\pi/n) + \cos 2k\pi/n)
\]
and its height is $\sin 2k\pi/n$, so the modulus of the $k$th cylinder is 
\[
\frac{2(1 + 2\cos 2\pi/n  + 2\cos 4\pi/n +\cdots + 2\cos(2(k-1)\pi/n) + \cos 2k\pi/n)}{\sin 2k\pi/n} = 2 \cot \pi/n
\]
by applying Lemma \ref{identities}. This says that the modulus does not depend on $k$, and is always $2 \cot \pi/n$. Thus, the effect of the matrix $M_n$ with $2 \cot \pi/n$ as its upper-right entry is a twist in each cylinder, and $M_n$ is in the Veech group.\end{proof}

\subsection{Cutting sequences on the square torus \label{square}}
As a prelude to our discussion of the regular pentagon and double odd-gons in general, we explain the geometric shearing process and its combinatorial effect on cutting sequences for the square torus, as was also done in the introduction of \cite{SU}. The purpose is to attain an understanding of the process for a simple case, before moving on to double polygons, where the pictures and methods appear more complicated. The square torus is perhaps the simplest example of a translation surface.

For the square torus, we obtain a surface by identifying opposite edges of a square in the usual way. Assume that the square has an edge with endpoints $(0,0)$ and $(1,0)$, and lies in the first quadrant. Then the Veech group of the square torus is $SL_2(\mathbb{Z})^{\pm}$, the group of affine automorphisms that preserve area, and take vertices of the square torus to integer points, and allow matrices with determinant $\pm 1$. (The classical definition of Veech groups is only orientation-preserving automorphisms, but we will allow  orientation-reversing maps because the shear we will work with is orientation-reversing.)

As an example, choose the matrix $\bigl( \begin{smallmatrix} 1&3\\ 0&1 \end{smallmatrix} \bigr) \in SL_2(\mathbb{Z})$. Then the square torus shears to the parallelogram in Figure \ref{parallelogram}. To see that this is what happens, simply apply the matrix to each of the square's vertices. The lower-right corner of the square has coordinates $(1,0)$ and stays at $(1,0)$; the upper-left corner has coordinates $(0,1)$ and moves to the point $(3,1)$, and so on. We can cut up the parallelogram along the gridlines, and reassemble it back into a square torus by translating the pieces. 

\begin{figure}[!h] 
\centering
\includegraphics[height=90pt]{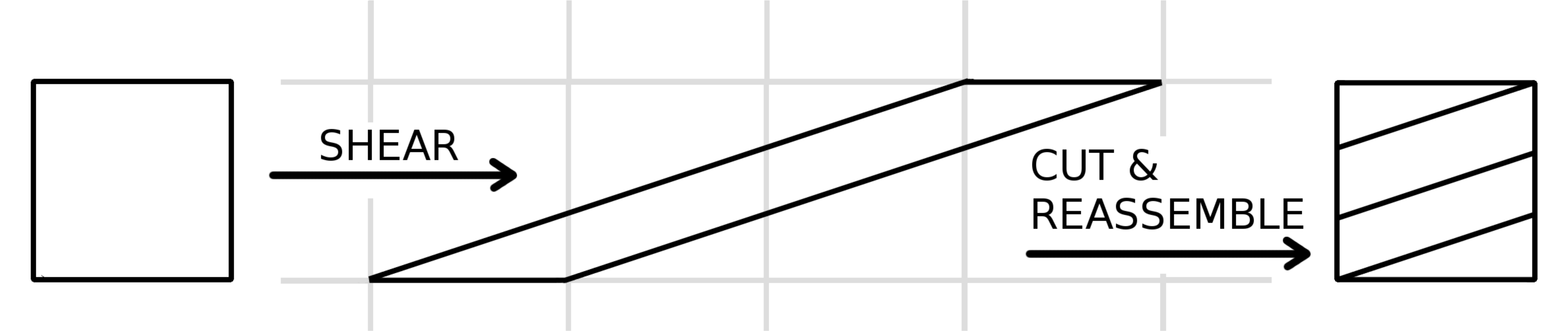}
\begin{quote}\caption{A square, sheared into a parallelogram, then cut and reassembled back into a square \label{parallelogram}} \end{quote}
\end{figure}

We should check that this process of shearing and reassembling preserves the side gluings that make the square into a toroidal surface. In fact, it does, as we can easily check in Figure \ref{parallelogramcheck}.

\begin{figure}[!h] 
\centering
\includegraphics[width=400pt]{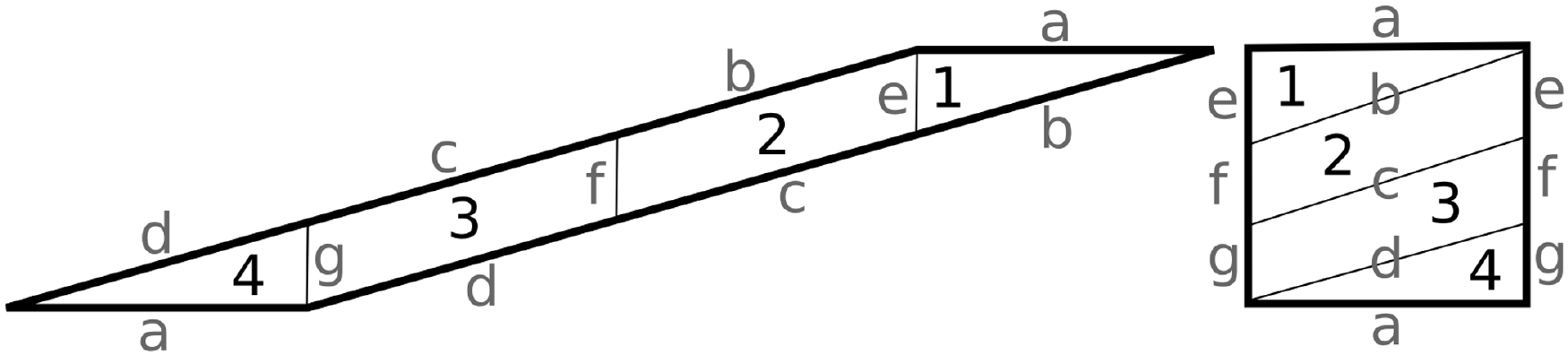}
\begin{quote}\caption{Checking the side identifications of the parallelogram \label{parallelogramcheck}} \end{quote}
\end{figure}

Now that we have described the geometric effect of applying an element of the Veech group to the square torus, we want to look at its geometric effect on the geodesic trajectories on these surfaces, and its combinatorial effect on the associated cutting sequences. We call the sequence resulting from applying the shear the \emph{derived sequence}. 

Because of the symmetry of the edge labels on the square torus, we will restrict our attention to trajectories with positive slopes, that travel from left to right (i.e., for $\theta$ the angle between the trajectory and the horizontal, $\theta \in [0,\pi/2)$). When the trajectory intersects an edge, we write down the letter ($A$ or $B$) associated with the edge. For example, the sequence illustrated at the extreme left of Figure \ref{squareshear} is $ABBB$, a periodic sequence of period $4$. 

The method for deriving a sequence on the square torus, using the Veech element $M_4^{-1} = \bigl( \begin{smallmatrix} 1&-1\\ 0&1 \end{smallmatrix} \bigr)$, is illustrated in Figure \ref{squareshear}: First, draw the trajectory ABBB on the square torus. Then cut the square along the dotted line and reassemble it into a parallelogram. Finally, "un-shear" the parallelogram back into a square by appling the inverse Veech element $M_4^{-1} =  \bigl( \begin{smallmatrix} 1&-1\\ 0&1 \end{smallmatrix} \bigr)$ to obtain the derived trajectory ABB.

\begin{figure}[!h]
\centering
\includegraphics[height=80pt]{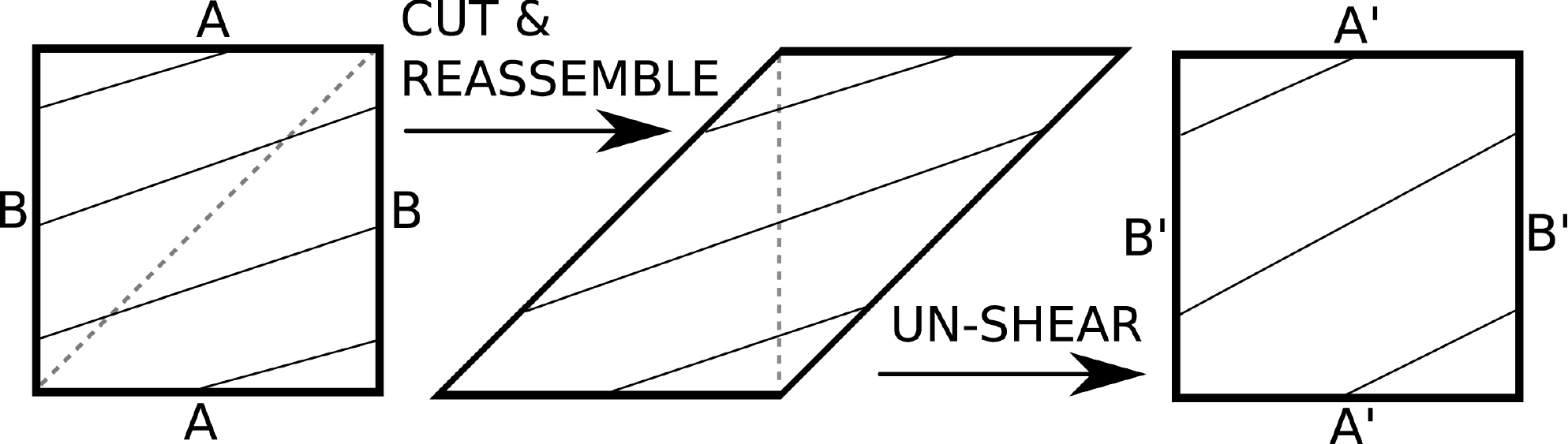}
\begin{quote}\caption{Our method for shearing the square torus and understanding the shear's effect on trajectories \label{squareshear}} \end{quote}
\vspace{-2.0em}
\end{figure}

Other literature uses the four-step process illustrated in Figure \ref{squareshear2nd}, using the matrix $M_4 = \bigl( \begin{smallmatrix} 1&1\\ 0&1 \end{smallmatrix} \bigr)$: First, draw the trajectory $ABBB$ on the square torus. Then flip the square horizontally. Then shear the square to a parallelogram by applying the Veech element $M_4$. Make a vertical cut and reassemble the pieces back into a square. Finally, flip the square to return it to its original orientation, and obtain the derived trajectory $ABB$

In fact, either method is fine; we choose the first method because it is simpler. To see this method applied to the double pentagon, see Figure \ref{second123} in Chapter \ref{deriveBECE}.

\begin{figure}[!h]
\centering
\includegraphics[width=430pt]{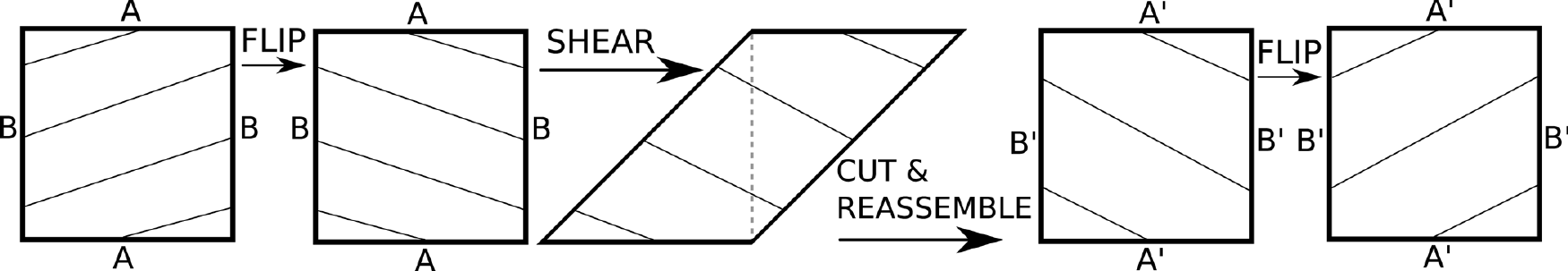}
\begin{quote}\caption{An equivalent method for determining the geometric effect of derivation, which we choose not to use \label{squareshear2nd}} \end{quote}
\vspace{-2.0em}
\end{figure}

There is an easy way to remember what happens to cutting sequences on the square under derivation: "Between every two A's, remove a B." Then it is straightforward to see that the derivation of the sequence ABBB is ABB; simply write out the sequence a few times. This rule applies to all cutting sequences on the square torus with $\theta \in [0,\pi/4)$ regardless of complexity; see \cite{Series}, or  Proposition 1.2.2 in \cite{SU}. Note that while the derivation rule $ksl(s)$ applies to sequences $s$ on the octagon, the double pentagon, and all double odd-gons, it does not apply to the square torus.

\subsection*{Acknowledgements}

I would like to thank Sergei Tabachnikov for illuminating conversations in December 2010; Ronen Mukamel for a helpful conversation in May 2011; Corinna Ulcigrai for helpful comments on an early draft; and Richard Schwartz for ongoing direction and guidance, and many useful suggestions as the paper reached a final form.

I would also like to thank the referee, whose very extensive and helpful report has considerably improved both the content and the form of this paper.

\section{The vertex generator guide \label{vertexguide}}

In this chapter, we will describe how to construct the vertex generator guide for a double $n$-gon for odd $n$. Our goal is to determine the new locations of the double $n$-gon's vertices, after shearing them with the Veech element $M_n =  \bigl( \begin{smallmatrix} 1&2\cot{\pi/n}\\ 0&1 \end{smallmatrix} \bigr)$. 

Let $P_U$ be a regular $n$-gon with unit-length edges, whose lowest vertices are $(0,0)$ and $(0,1)$ (shown as black dots in Figure \ref{vgglabeled}). Label the horizontal edge $S_1$, and label the edges $S_2, S_3, \ldots , S_n$ counter-clockwise. Let $P_L$ be the reflection of $P_U$ across $S_n$, and label its edges to correspond to those of $P_U$ (see Figure \ref{vgglabeled}).

\begin{figure}[!h]
\centering
\includegraphics[width=300pt]{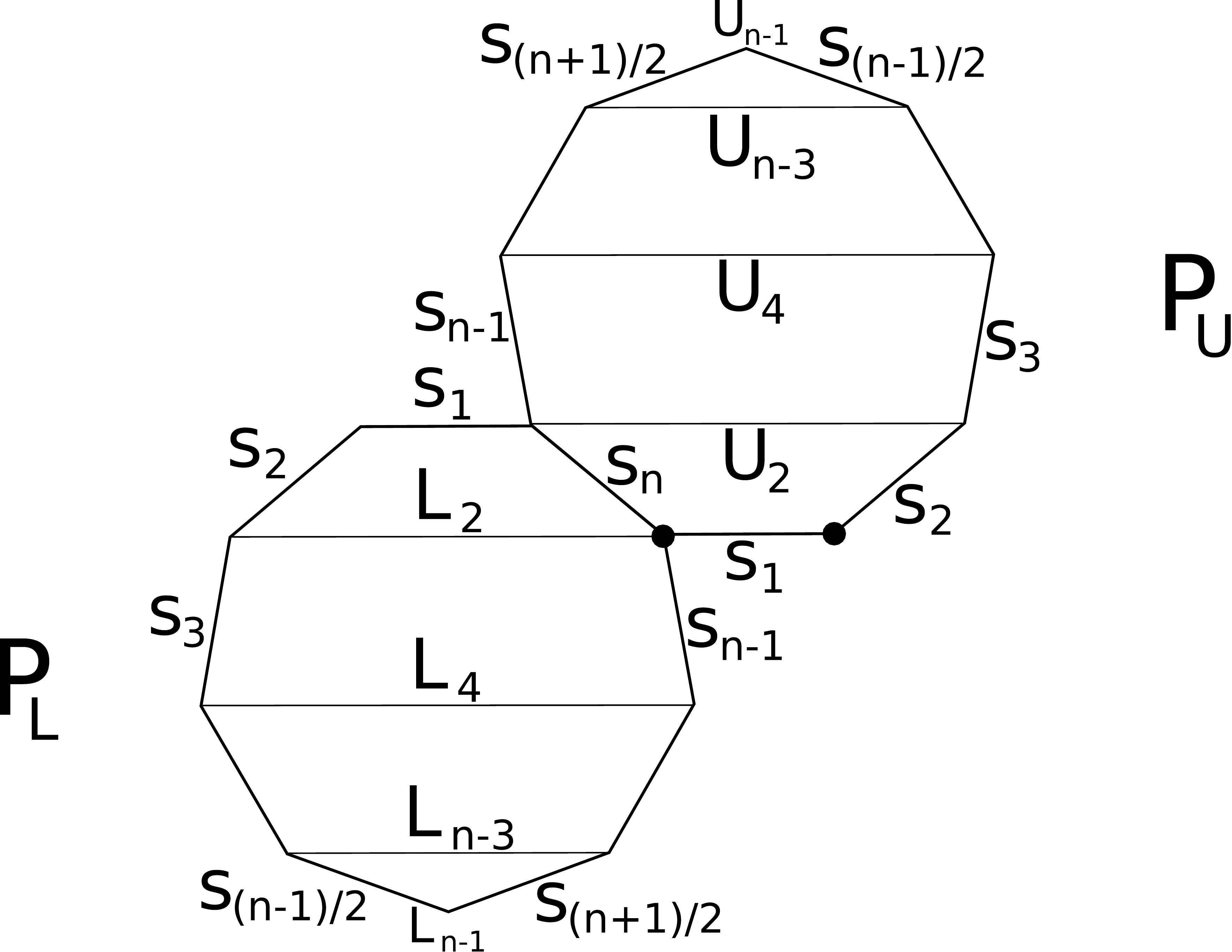}
\begin{quote}\caption{Setting up the polygon, edge and horizontal bar labels for the vertex generator guide \label{vgglabeled}} \end{quote}
\vspace{-2.0em}
\end{figure}

We label the polygons' horizontal diagonals $U_2, U_4, \ldots, U_{n-3}$, $L_2, L_4, \ldots , L_{n-3}$. To avoid awkward terminology, we will refer to these segments, and certain of their translated images, as \emph{horizontal bars}. It is helpful to think of the lowest and highest vertices as horizontal bars of length $0$, so we will label them $U_{n-1}$ and $L_{n-1}$ for consistency. The vertex generator guide is defined based on these bars.

First, we will construct the images of $P_U$. Glue a copy of $P_L$ along edge $S_2$ of $P_U$. Then glue a copy of $P_U$ along edge $S_n$ of this right-most copy of $P_L$, and denote it by $P_U^{(2)}$ (Figure \ref{vggU2}).

\begin{figure}[!h]
\centering
\includegraphics[width=250pt]{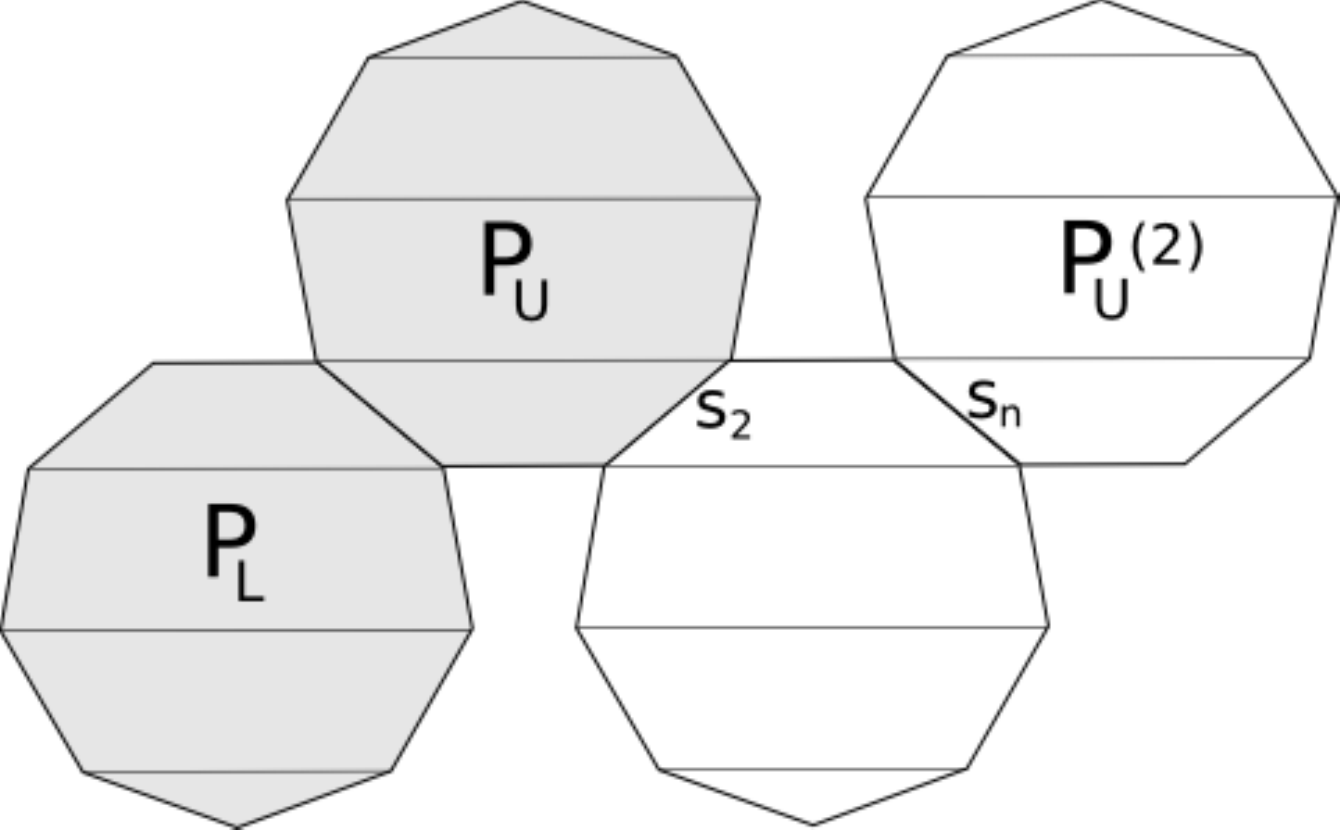}
\begin{quote}\caption{Constructing $P_U^{(2)}$. The original polygons are shaded in grey \label{vggU2}} \end{quote}
\vspace{-2.0em}
\end{figure}

Next, glue a copy of $P_L$ along edge $S_3$ of $P_U^{(2)}$. Then glue a copy of $P_U$ along edge $S_{n-1}$ of the right-most copy of $P_L$, and denote it by $P_U^{(4)}$ (Figure \ref{vggU4}).

\begin{figure}[!h]
\centering
\includegraphics[width=300pt]{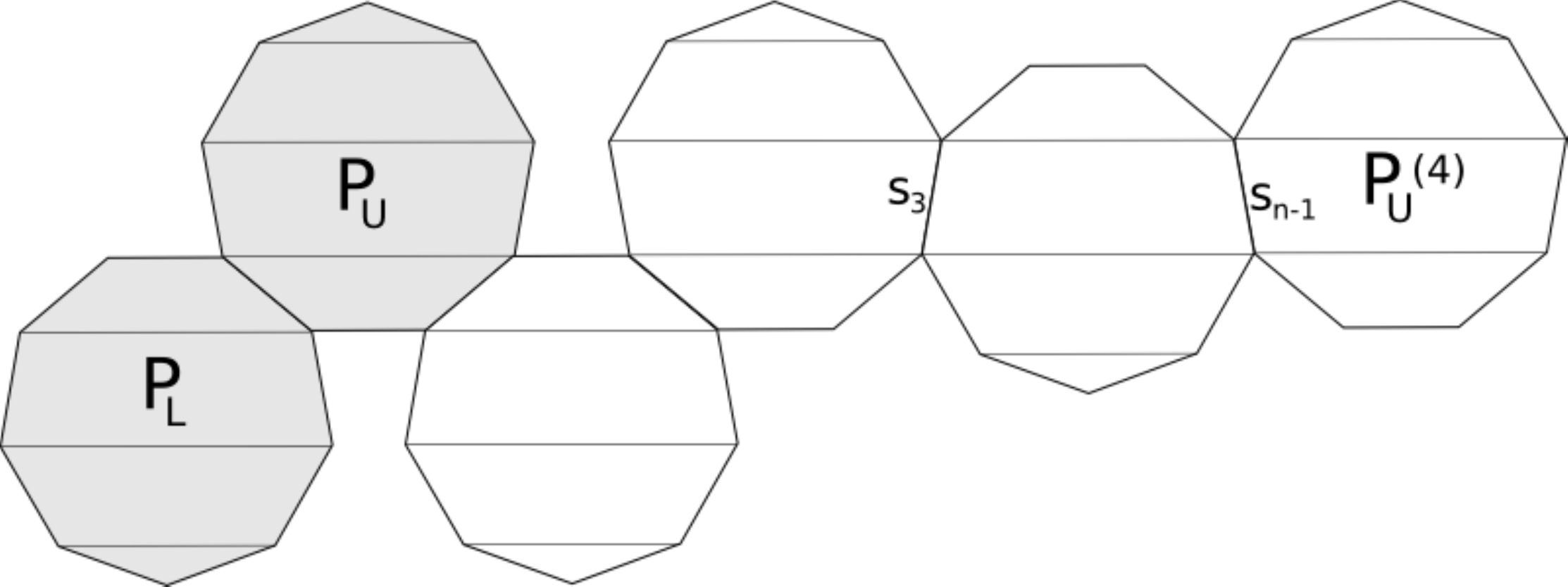}
\begin{quote}\caption{Constructing $P_U^{(4)}$. The original polygons are shaded in grey \label{vggU4}} \end{quote}
\vspace{-2.0em}
\end{figure}

Continue in this manner, gluing a copy of $P_L$ along edge $S_k$ of the right-most copy of $P_U$ and then gluing a copy of $P_U$ along edge $S_{n+2-k}$ of the right-most copy of $P_L$, and denoting it by by $P_U^{(2(k-1))}$. (It is essential to proceed in order, $k=3, 4, \ldots, (n-1)/2$.) This continues until $k=(n-1)/2$ (considering the top vertex of $P_U$ to be a horizontal bar, $U_{n-1}$, of length $0$), at which point we cannot go any further and we stop. The full construction for all the images of $P_U$ for $n=9$ is shown in Figure \ref{vggPUall}.

\begin{figure}[!h]
\centering
\includegraphics[width=430pt]{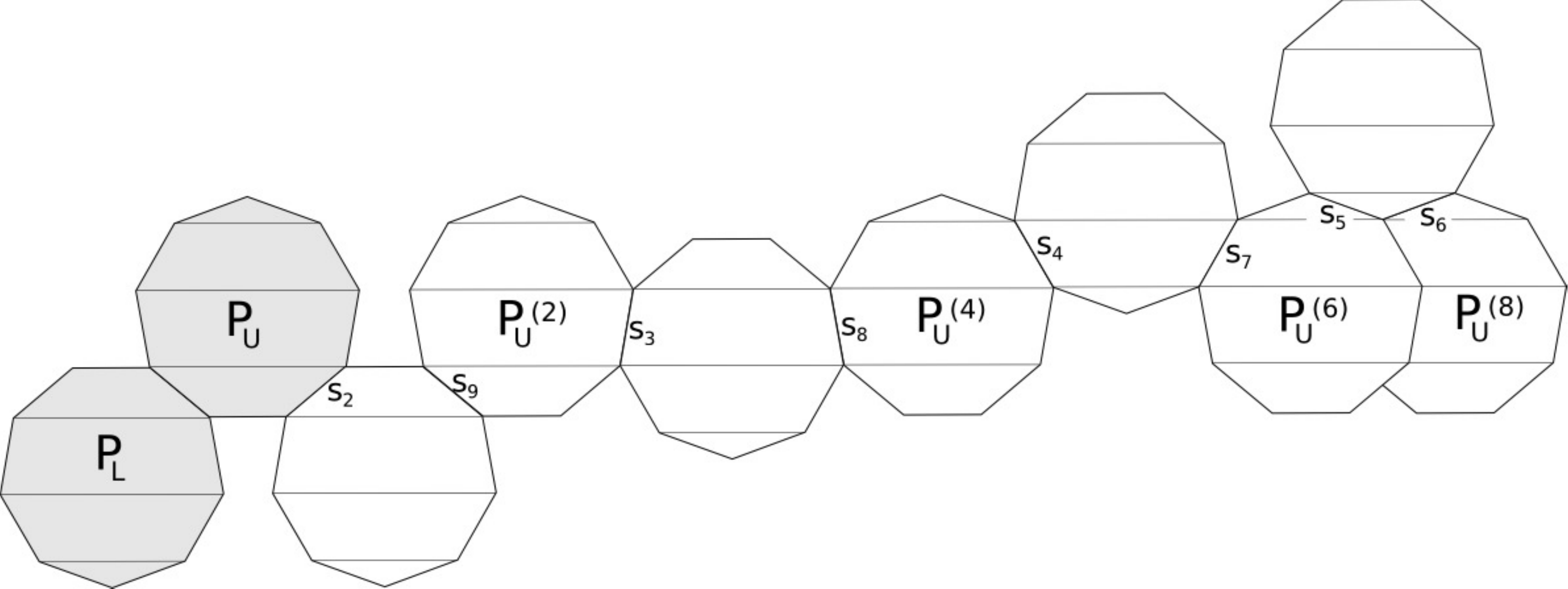}
\begin{quote}\caption{Constructing all the images of $P_U$ \label{vggPUall}} \end{quote}
\vspace{-2.0em}
\end{figure}

\begin{remark} \label{whatabouteven}
The construction is similar for single even-gons. In that case, $P_U$ is a $2n$-gon with edges $S_1,\ldots,S_n,S_1,\ldots,S_n$. $P_L$ is a a copy of $P_U$, so at this point of the construction we are done. If we used a double even-gon, and $P_L$ were considered a different polygon than $P_U$, we would need to continue with the construction of its image, below.
\end{remark}

Now we will construct the images of $P_L$. Bar $L_2$ lies along the $x$-axis, so it is fixed under a horizontal shear.

Glue a copy of $P_U$ along edge $S_3$ of $P_L$. Then glue a copy of $P_L$ along edge $S_{n-1}$ of this right-most copy of $P_U$, and denote it by $P_L^{(4)}$.

Next, glue a copy of $P_U$ along edge $S_4$ of $P_U^{(4)}$. Then glue a copy of $P_L$ along edge $S_{n-2}$ of the right-most copy of $P_U$, and denote it by $P_L^{(6)}$.

Continue in this manner, gluing a copy of $P_U$ along edge $S_k$ of the right-most copy of $P_L$ and then gluing a copy of $P_L$ along edge $S_{n+2-k}$ of the right-most copy of $P_U$, and denoting it by by $P_L^{(2(k-1))}$. (It is essential to proceed in order, $k=3, 4, \ldots, (n-1)/2$.) This continues until $k=(n-1)/2$ (considering the bottom vertex of $P_L$ to be a horizontal bar, $L_{n-1}$, of length $0$), at which point we cannot go any further and we stop. The full construction for all the images of $P_L$ for $n=9$ is shown in Figure \ref{vggPLall}.

\begin{figure}[!h]
\centering
\includegraphics[width=370pt]{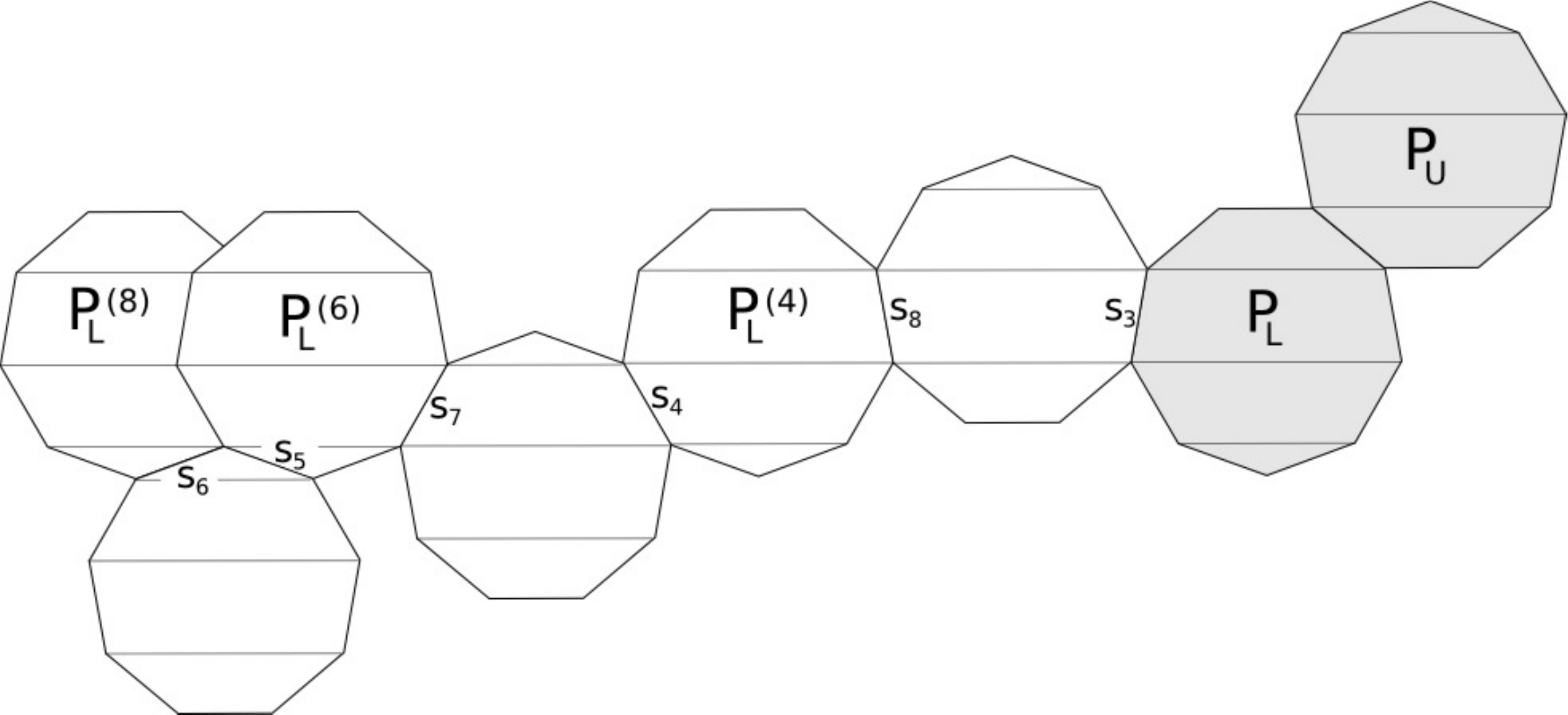}
\begin{quote}\caption{Constructing all the images of $P_L$ \label{vggPLall}} \end{quote}
\vspace{-2.0em}
\end{figure}

Examples of the complete construction for $n=5, 7$ and $9$ are shown in Figures \ref{pentshear}-\ref{nonagonshear} and the construction for $P_U$ is shown for $n=11$ in Figure \ref{11gongrid}.

\begin{definition} \label{vgg}
Denote by $U_k$' the bar in $P_U^{(k)}$ that is contained in the same horizontal line as $U_k$, and denote by $L_k$' the bar in $P_L^{(k)}$ that is contained in the same horizontal line as $L_k$ (Figure \ref{wholevgg}). The \emph{vertex generator guide} is a collection of points in the plane that are the endpoints of the horizontal bars $U_2', U_4', \ldots, U_{n-1}', L_2', L_4', \ldots , L_{n-1}'$ (shown as black dots in Figure \ref{wholevgg}).
\end{definition}

\begin{figure}[!h]
\centering
\includegraphics[width=430pt]{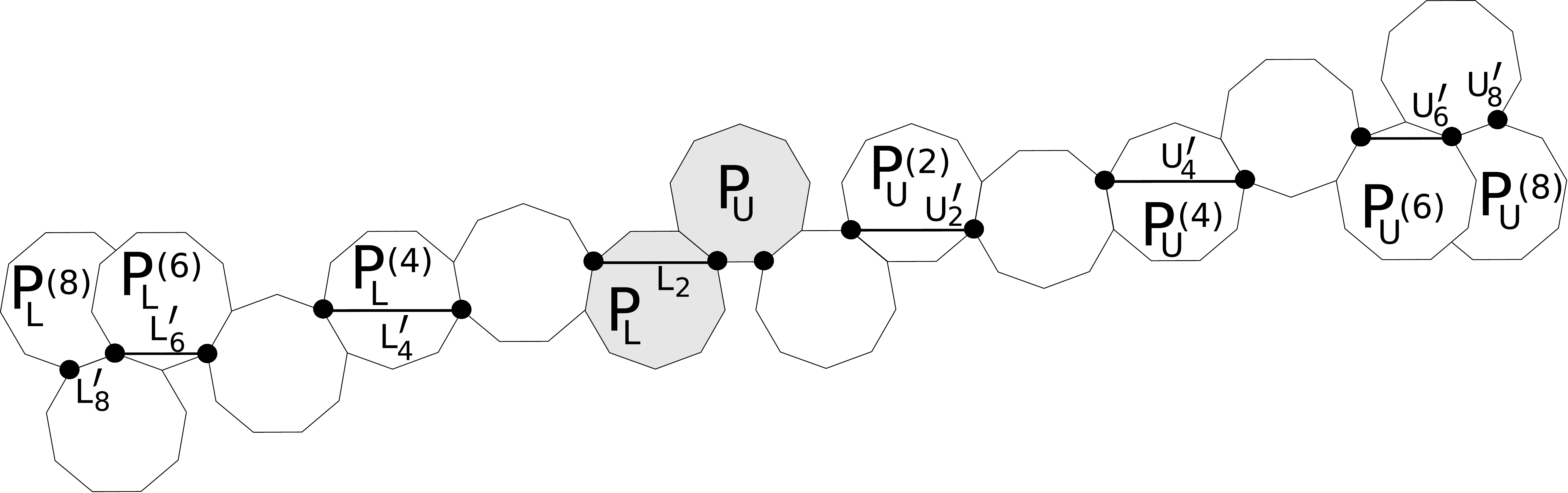}
\begin{quote}\caption{The vertex generator guide (the example shown is for $n=9$) is the collection of points labeled by black dots \label{wholevgg}} \end{quote}
\vspace{-2.0em}
\end{figure}

We have only described how to construct the sheared images of the polygons; we have not proven that this is actually the effect of applying the Veech element $M_n$. We will prove that this method actually constructs the locations of the sheared polygons' vertices in general for all $n$ in Chapter \ref{reassembly}, by writing down the vertices' coordinates and transforming them by the shearing matrix $M_n$.

\section{Derivation and cutting sequences on the double pentagon\label{deriveBECE} \label{veechofpentagon}}

As explained in the introduction, we'll carefully go through the case of the double pentagon, explaining the geometric effect of applying the shearing Veech element on the pentagons, and working out the combinatorial effect on the cutting sequences. We'll use the same method for double $n$-gons in general, so it will be helpful to work out the pentagon case, where we can explicitly draw out all of the pictures and see exactly what is happening. 

As we have seen, since a regular pentagon has no parallel sides, in order to create a surface we have to use two copies of the pentagon, which are reflections of each other, as in Figure \ref{pentagonshear}. An Euler characteristic argument shows that this surface has genus $2$. This is the same as for the octagon: If we glue together one pair of identified sides, we can see that the double pentagon is an equilateral but non-equiangular octagon.

We will call the double pentagon on the left in Figure \ref{pentagonshear} the \emph{standard double pentagon}: the pentagon on the left has the horizontal side $A$ on the top; the pentagon on the right has the slanted sides $D$ and $C$ on the top; and the two pentagons are glued together along side $E$.

We apply the flipping and shearing Veech element to the standard double pentagon (Figure \ref{pentagonshear}). The white pentagon shears to the elongated white pentagon, and the gray pentagon shears to the elongated gray pentagon. The edges shear to the primed edges, with side labels exchanged because of the horizontal flip.

\begin{figure}[!h] 
\centering
\includegraphics[width=430pt]{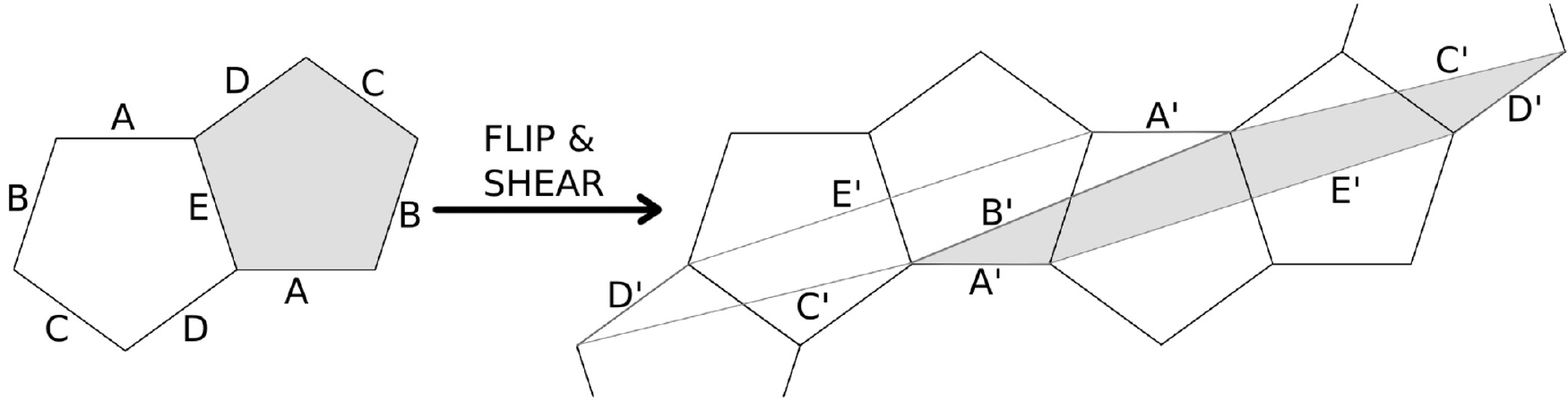}
\begin{quote}\caption{Applying the Veech element to the standard double pentagon\label{pentagonshear}} \end{quote}
\end{figure}

We apply the flipping and shearing Veech element $ \bigl( \begin{smallmatrix} -1&0\\ 0&1 \end{smallmatrix} \bigr) M_5$ to the double pentagon, obtaining the sheared pentagons in Figure \ref{pentagonshear}. The effect of $ \bigl( \begin{smallmatrix} -1&0\\ 0&1 \end{smallmatrix} \bigr)$ is a left-right exchanging of the edge labels, reversing the orientation. To cut up the sheared pentagons and reassemble them back into the standard double pentagon, we cut along the edges of the pentagonal vertex generator guide. We can reassemble them as in Figure \ref{puzzlelocations}. 

\begin{figure}[!h]
\centering
\vspace{-1.0em}
\includegraphics[width=305pt]{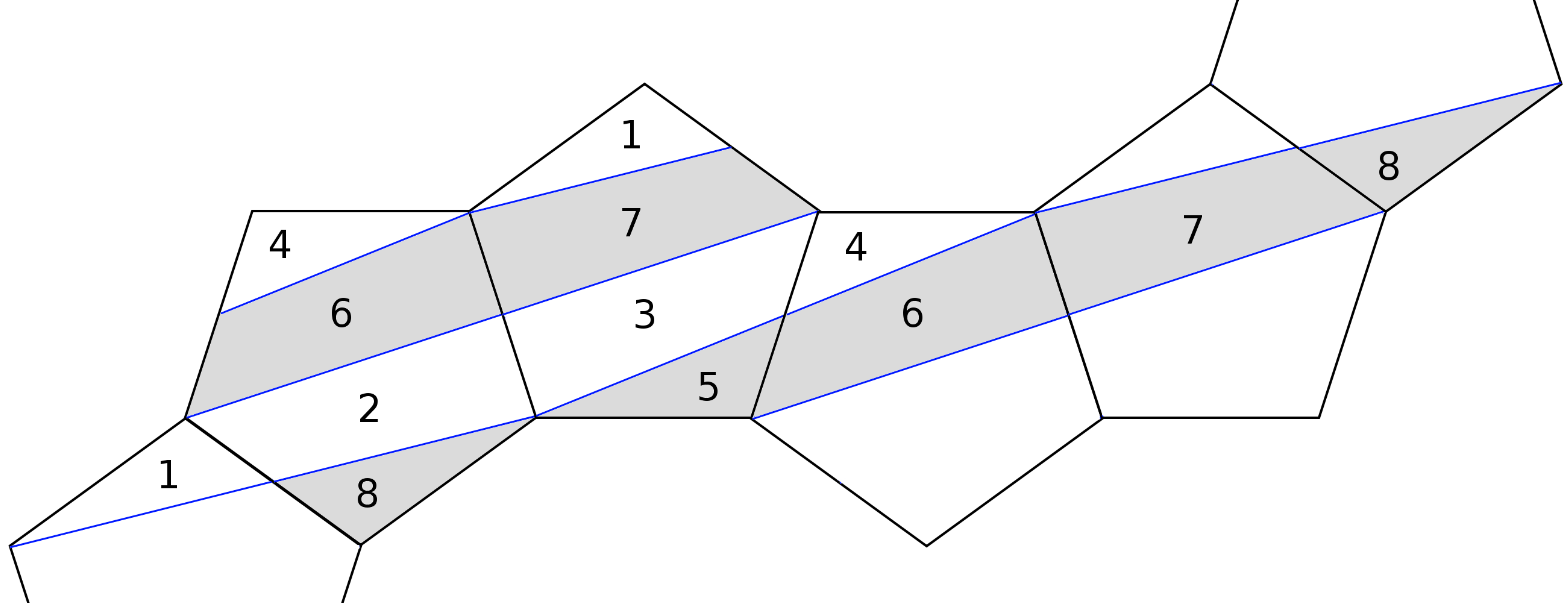}
\begin{quote}\caption{The sheared white and gray pentagons, whose parts are consecutively labeled 1-4 and 5-8, can be cut up and reassembled into a standard double pentagon \label{puzzlelocations}} \end{quote}
\vspace{-1.0em}
\end{figure}

Note that no other cutting is allowed; we may only cut the sheared pentagons along the edges of the regular pentagons that underly the pentagonal vertex generator guide. Also, when we reassemble the pieces, we must translate them; rotations of individual pieces do not preserve the surface. Similarly, we may only cut the regular pentagons along the edges of the sheared pentagons, as the standard double pentagon is cut up here. (We emphasize this point because any polygons of the same area are cut-and-paste equivalent, but here both the cutting and pasting operations are very restricted.)

Now we will explore derivation on the double pentagon, in two different ways. In section \ref{geometricprocess}, we will demonstrate the geometric process of derivation on the double pentagon. Then in section \ref{sequences}, we will show that the combinatorial effect on a cutting sequence $s$ on the double pentagon is $ksl(s)$. We will use the sequence $BECE$ as our example in both sections. The same procedure will generalize to double $n$-gons for all odd $n$, in Chapter \ref{alln}. 

\subsection{The geometric process of derivation \label{geometricprocess}}

This 4-step process illustrates the geometric effect of applying a shearing Veech element to a double pentagon, and to a trajectory on the double pentagon surface.

\vspace{10pt} \noindent $\mathbf{STEP \ 1}$: First, we'll draw our trajectory (in this case, $BECE$) on the standard double pentagon (Figure \ref{second123} (a)). 

\vspace{10pt} \noindent $\mathbf{STEP \ 2}$: We cut up the double pentagon along the primed edges, and reassemble the pieces into the sheared double pentagon (Figure \ref{second123} (b)), as we showed in Figure \ref{puzzlelocations}. We have not changed the slope of the trajectory at all, simply reassembled it. 

\vspace{10pt} \noindent $\mathbf{STEP \ 3}$: We "un-shear" the sheared pentagons: we apply the inverse of the shearing matrix $\bigl( \begin{smallmatrix} 1&2\cot(\pi/5)\\ 0&1 \end{smallmatrix} \bigr)$ to obtain a trajectory on the standard double pentagon (Figure \ref{second123} (c)).

\begin{figure}[!h] 
\centering
\includegraphics[width=360pt]{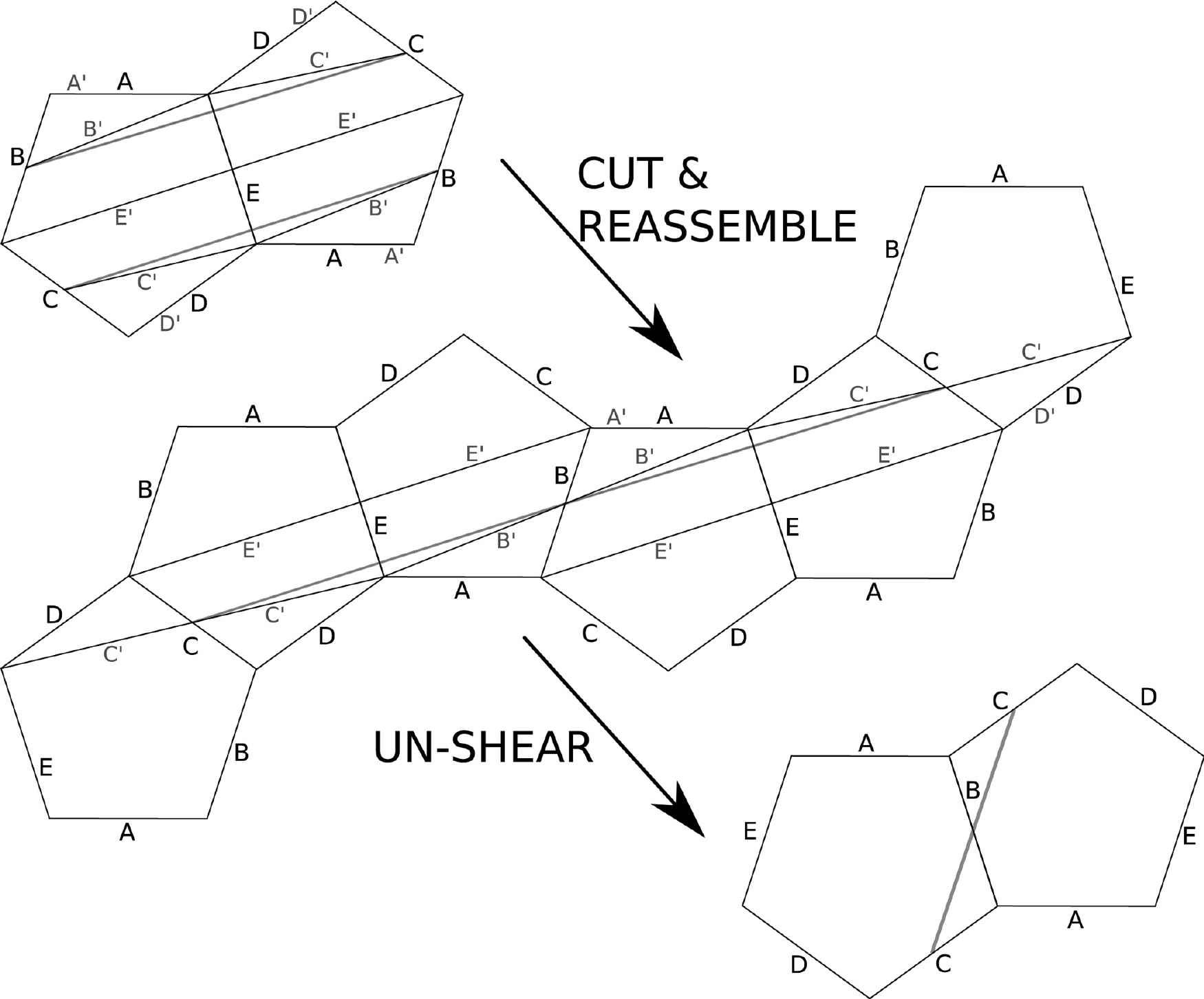}
\begin{quote}\caption{Steps 1-3 of the geometric process of derivation\label{second123}} \end{quote}
\end{figure}

\begin{remark}
To do this in practice, all we do is see where the trajectory crosses the primed edges, and draw the trajectory in the corresponding places on the original edges. In this case, the trajectory goes from the midpoint of side $C'$, to the midpoint of side $B'$, to the midpoint of side $C'$. So on the double pentagon, we simply draw a trajectory from the midpoint of $C$ to the midpoint of $C$, passing through the midpoint of $B$ (Figure \ref{second123} (c) and Figure \ref{second345} (a)). Notice that these pentagons are not in the standard orientation, because while the pentagon on the left has its horizontal edge on top, the shared edge in the middle is $B$, not $E$, and the edges at the top of the pentagon on the right are $C$ and $D$, rather than $D$ and $C$.
\end{remark}

\begin{figure}[!h] 
\centering
\includegraphics[width=430pt]{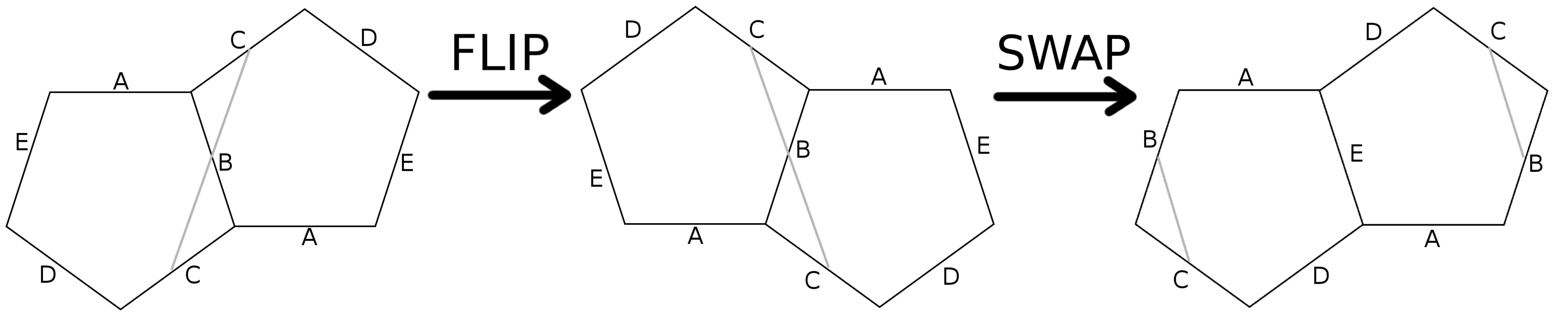}
\begin{quote}\caption{Step 4 of the geometric process of derivation\label{second345}} \end{quote}
\vspace{-2.0em}
\end{figure}

\vspace{10pt} \noindent $\mathbf{STEP \ 4}$: To put this trajectory back onto the standard double pentagon, we perform a horizontal flip (applying the orientation-reversing matrix $\bigl( \begin{smallmatrix} -1&0\\ 0&1 \end{smallmatrix} \bigr)$) and then rearrange the two pentagons to the standard orientation, as in Figure \ref{second345}. This gives us a trajectory corresponding to the sequence BC.

\begin{remark}\label{lookatprimes}
Notice that we didn't need to do the last step in order to see that the derived sequence is $BC$; we could have seen that easily in Figure \ref{second345} (a), because flipping and swapping the pentagons didn't affect the sequence. In fact, we could already see it in Figure \ref{second123}, simply by looking at which "primed" edges the trajectory crossed. It crossed edges $B'$ and $C'$, so the derived trajectory is $BC$. Our process going forward will be to draw in the primed edges on the original double pentagon and see which primed edges our trajectory crosses; this will give us the derived sequence.
\end{remark}

\subsection{The effect on symbolic sequences\label{sequences}}

One can determine that the derivation of the sequence $BECE$ is $BC$ without having to draw out multiple diagrams, as follows: As explained in the introduction, the derivation rule is $ksl(s)$ for a sequence $s$ on any double odd-gon (Theorem \ref{koslalln}); we will prove the case of the double pentagon in this section.

\begin{remark}
In the work that follows, we will assume that all trajectories have a direction $\theta \in [0,2\pi/5)$ (horizontal or positively-sloped, and traveling from left to right). A trajectory at an angle outside this sector can be transformed to a trajectory in this sector by applying an affine automorphism, whose derivative is a rotation matrix, which induces a permutation on the labels of the edges, and thus of the letters in the corresponding cutting sequence.
\end{remark}

We will now prove the Reassembly Theorem (\ref{theyrethesame}) for the special case of the double pentagon surface: \emph{On the double pentagon, the effect of the Veech element $M_5$ on a cutting sequence $s$ is $ksl(s)$.}

\begin{proof}
This proof will have $6$ steps. As noted in Remark \ref{lookatprimes}, it is possible to determine a derived trajectory by drawing the original trajectory on Figure \ref{withedges} and writing down the letters corresponding to the primed edges it crosses. We will determine which primed edges the trajectory crosses by going through the 6 steps, examining at each stage both the geometric effect on the pentagon and trajectory, and the combinatorial effect on the cutting sequence. 

\begin{figure}[!h]
\centering
\includegraphics[width=230pt]{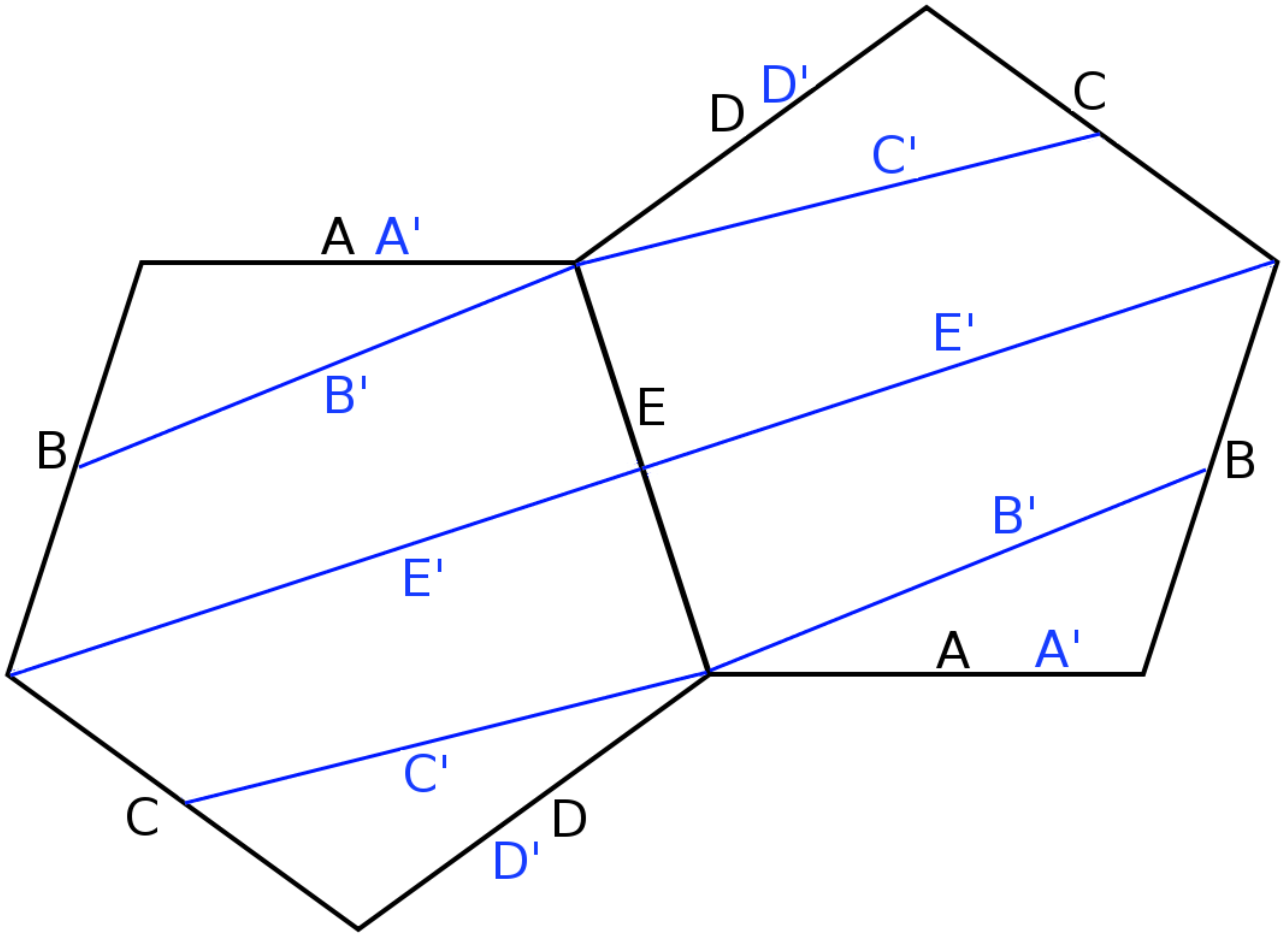}
\begin{quote}\caption{The standard double pentagon with primed edges drawn in\label{withedges}} \end{quote}
\end{figure}

It might appear trivial to determine the derived sequence; after all, one must merely observe which primed edges a trajectory crosses and write them down, and remove the ``primes" to obtain the derived sequence. However, an example illustrates that it is not obvious: 

\begin{example} \label{BEexample}
Suppose a given trajectory contains the two consecutive letters $BE$. The relevant question is whether it crosses edge $B'$ or edge $E'$ (refer to Figure \ref{withedges}). However, without knowing anything about the specific trajectory, it's impossible to know which primed edges it crosses: It might cross both $B'$ and $E'$, or neither. 
\end{example}

We will show how to add "auxiliary edges" that eliminate this uncertainty. The process described below is due to \cite{SU} for the regular octagon; we apply their process to the double pentagon.

\vspace{10pt} \noindent $\mathbf{STEP \ 1}$: First, draw a "transition diagram" of possible two-letter subwords of our sequence (Figure \ref{arrows}). Since the direction of the trajectory $\theta \in [0,\pi/5)$, our options are restricted. If a trajectory starts on side $A$, it can't go anywhere in the light pentagon (since trajectories go up), and on the dark pentagon, it can go to $B$. It can't go to $C$, because that would require $\theta > \pi/5$. So we draw this as a dark arrow from $A$ to $B$. The rest of the diagram is generated in the same way. Let's call this the "arrows" diagram.

\begin{figure}[!h]
\centering
\includegraphics[width=370pt]{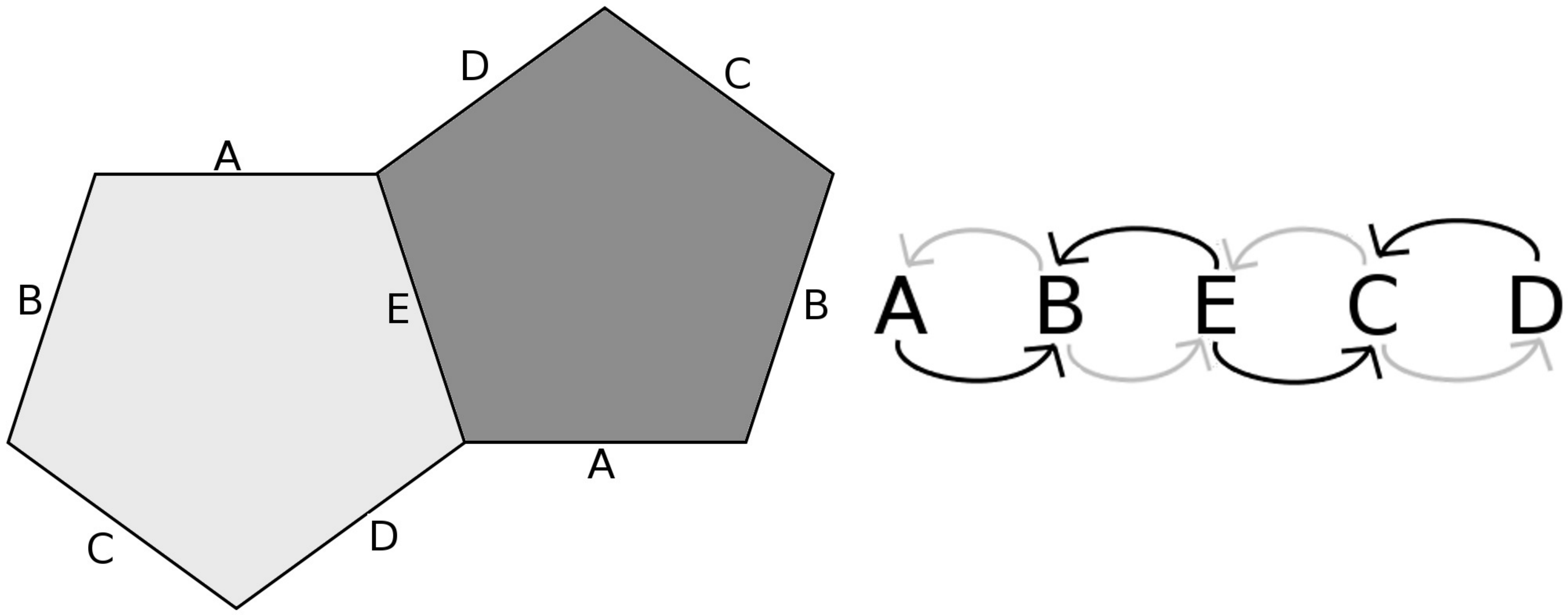}
\begin{quote}\caption{The "arrows" transition diagram for the double pentagon\label{arrows}} \end{quote}
\vspace{-2.0em}
\end{figure}

Now we can see a cutting sequence as a path on the transition diagram, following the arrows from letter to letter. A cutting sequence determines a unique path around the diagram, and a path on the transition diagram determines a unique cutting sequence. Note that any path on the diagram alternates black arrows and gray arrows. From this we can see that every valid cutting sequence on the double pentagon has an even period. We should also note that some paths on the transition diagram, such as $DCDCEC$ (and equivalently $ABABEB$), are not realizable as cutting sequences of trajectories on the double pentagon.

\vspace{10pt} \noindent $\mathbf{STEP \ 2}$: Now we add auxiliary diagonals at angles $\theta = 0,\theta = \pi/5$, labeling them $f$, $g$, $h$, and $i$ (Figure \ref{arrowsletters}). Now (recalling Example \ref{BEexample}), if a given sequence includes $\ldots BE\ldots$, the relevant question is whether its corresponding trajectory crosses edge $g'$, and it is clear that it must. So we label the arrow from $B$ to $E$ with the letter $g$, and similarly for $f$, $h$ and $i$. We can augment our sequence by inserting the letters for the auxiliary diagonals into the sequence itself, so $BECE$ becomes $BgEhCfEi$. A path on the arrows diagram uniquely determines a path on this "augmented arrows" diagram. 

\begin{figure}[!h]
\centering
\includegraphics[width=340pt]{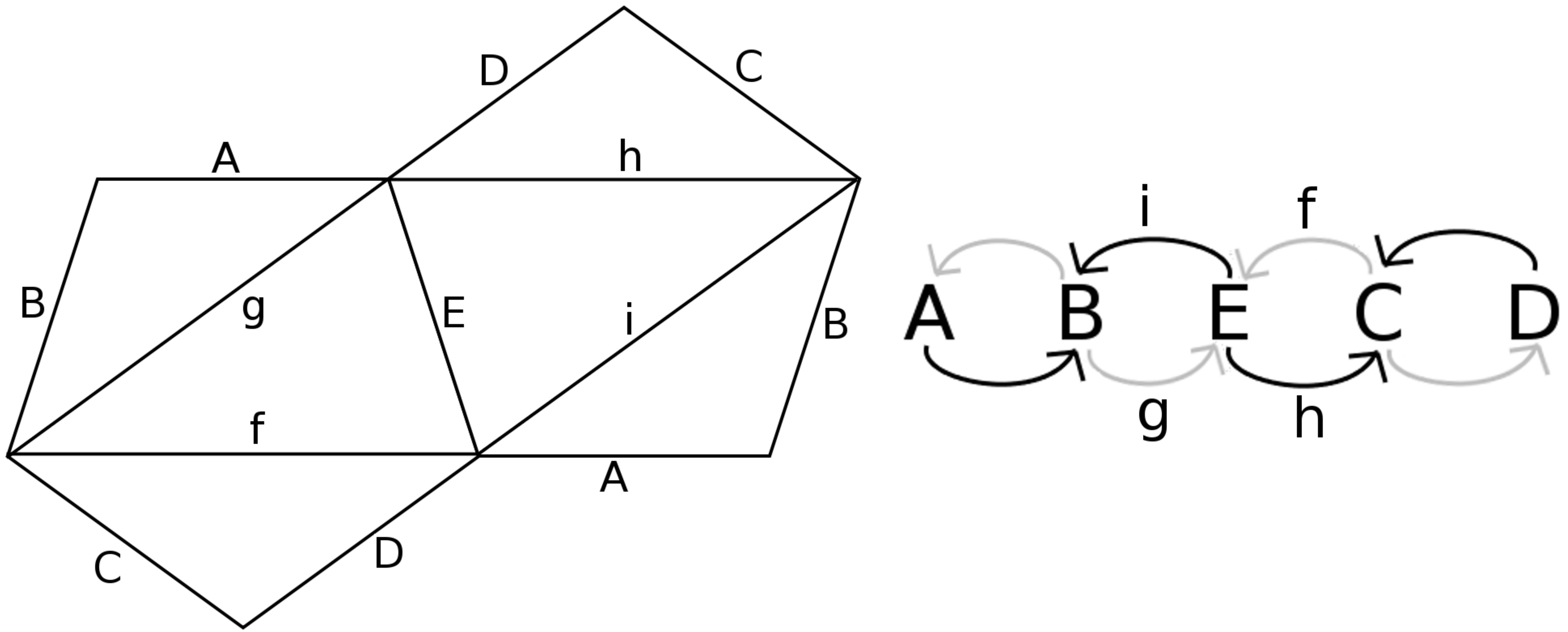}
\begin{quote}\caption{The double pentagon with auxiliary diagonals, and the "augmented arrows" transition diagram\label{arrowsletters}} \end{quote}
\vspace{-2.0em}
\end{figure}

\vspace{10pt} \noindent $\mathbf{STEP \ 3}$: Now we can "take the dual" of this diagram, making $A$, $D$, $f$, $g$, $h$, and $ i$ into the nodes and $B$, $C$, and $E$ into the arrow labels (Figure \ref{dualletters}). Our path on the augmented arrows diagram uniquely determines a path on this "dual" diagram, and our sequence is unchanged: it is still $BgEhCfEi$. Note that we have extra arrows on either side, because to get from side $A$ back to itself the trajectory must pass through $B$, and to get from side $D$ back to itself the trajectory must pass through $C$.

\begin{figure}[!h]
\centering
\includegraphics[width=430pt]{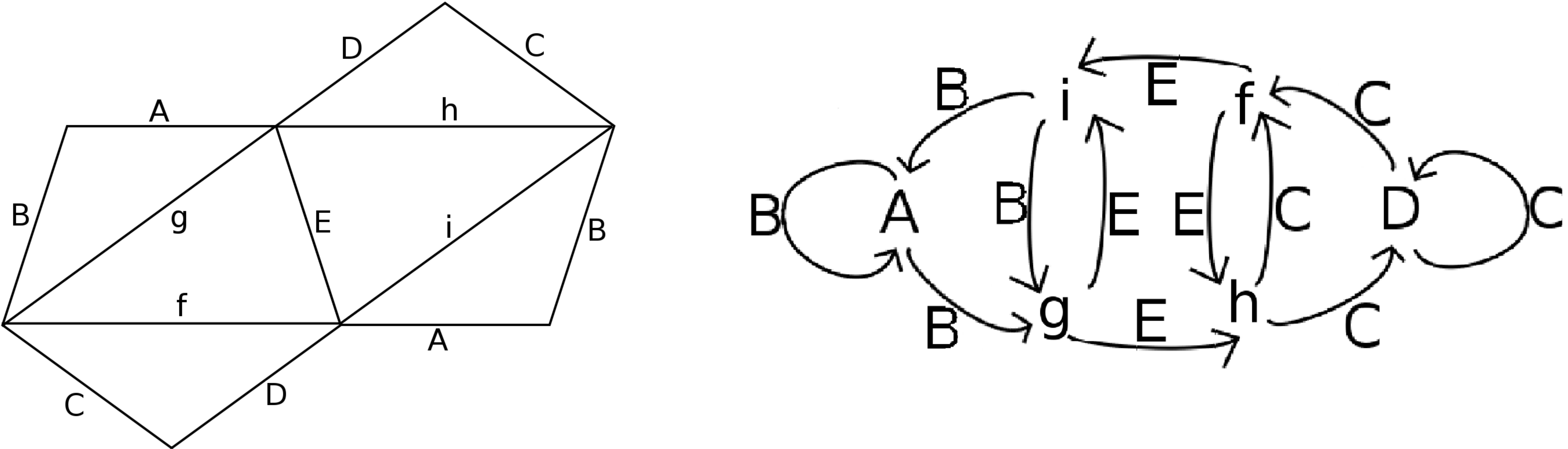}
\begin{quote}\caption{The "dual" transition diagram, with the sequence unchanged \label{dualletters}} \end{quote}
\vspace{-2.0em}
\end{figure}

\vspace{10pt} \noindent $\mathbf{STEP \ 4}$: Our goal is to end up with just the primed edges that our trajectory crosses. First, we eliminate all of the edges except those at $\theta = 0$ ($A$, $f$ and $h$) and at $\theta = \pi/5$ ($D$, $g$ and $i$). In the transition diagram, we keep only the nodes, which are the same six letters, and the arrows (Figure \ref{arrowsempty}). A path on the dual diagram corresponds uniquely to a path on this ``unlabeled" diagram. In the sequence, we keep only those six letters, so our sequence is now $ghfi$.

\begin{figure}[!h]
\centering
\includegraphics[width=350pt]{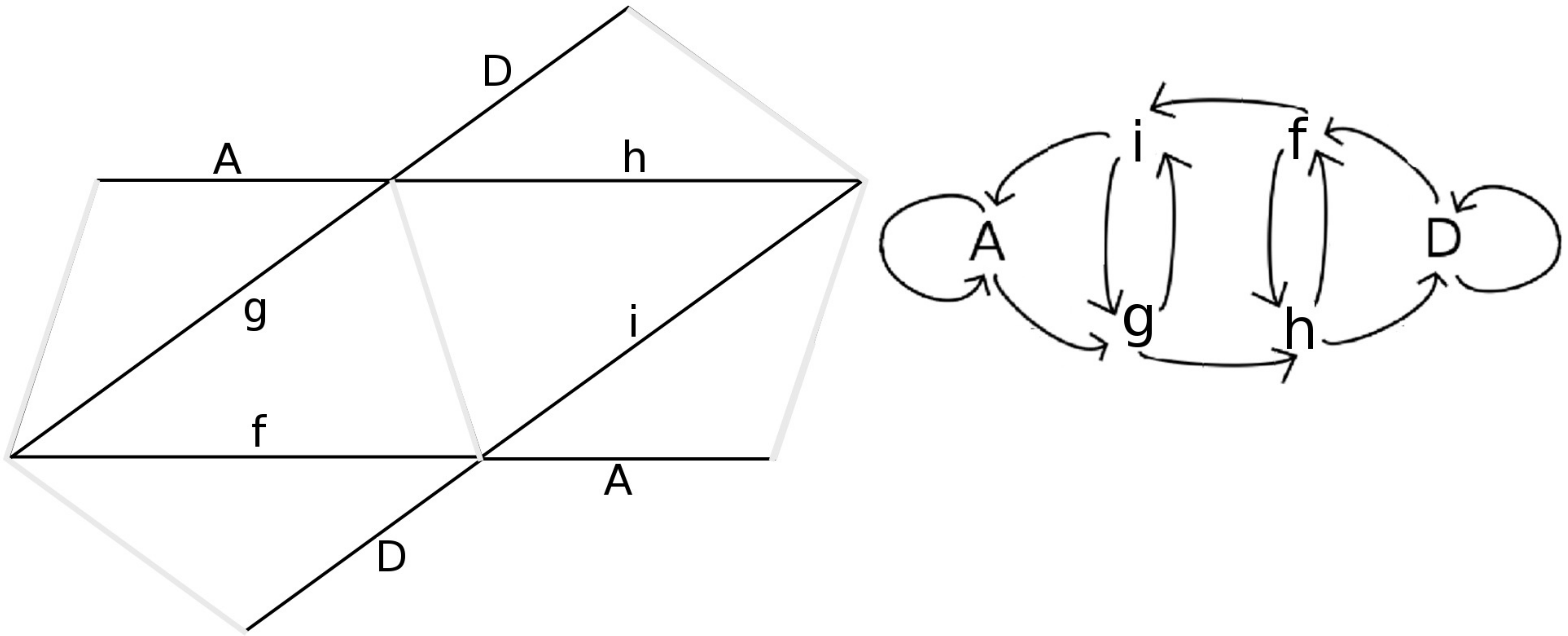}
\begin{quote}\caption{The double pentagon with only those edges at $\theta = 0$ and $\theta = 2\pi/5$, and the ``unlabeled" transition diagram with B, C, E deleted \label{arrowsempty}} \end{quote}
\vspace{-2.0em}
\end{figure}

\vspace{10pt} \noindent $\mathbf{STEP \ 5}$: Finally, we are ready to determine which primed edges our trajectory crosses, which will give us the derived sequence. We draw the pentagons with edges $A, D, f, g, h, i$ from before, and now add in the primed edges $A', B', C', D', E'$ (the outlined pentagons in Figure \ref{adjacency}). Note that side $A = A'$ and side $D = D'$, so we only label those sides with the primed letters, which eliminates the non-primed edges completely. 

Now we can label the arrows in our transition diagram with the primed edges that are crossed to get from one node to another. For instance, to get from edge $g$ to edge $h$, no other edges are crossed, so the arrow is unlabeled. To get from edge $f$ to edge $h$, we cross edge $E'$. The rest of the arrows can be labeled or left empty in the same manner. To do this, it is helpful to draw out enough pentagons so that each of the primed edges appears as a diagonal in some parallelogram (Figure \ref{adjacency}). 

\begin{figure}[!h]
\centering
\includegraphics[width=430pt]{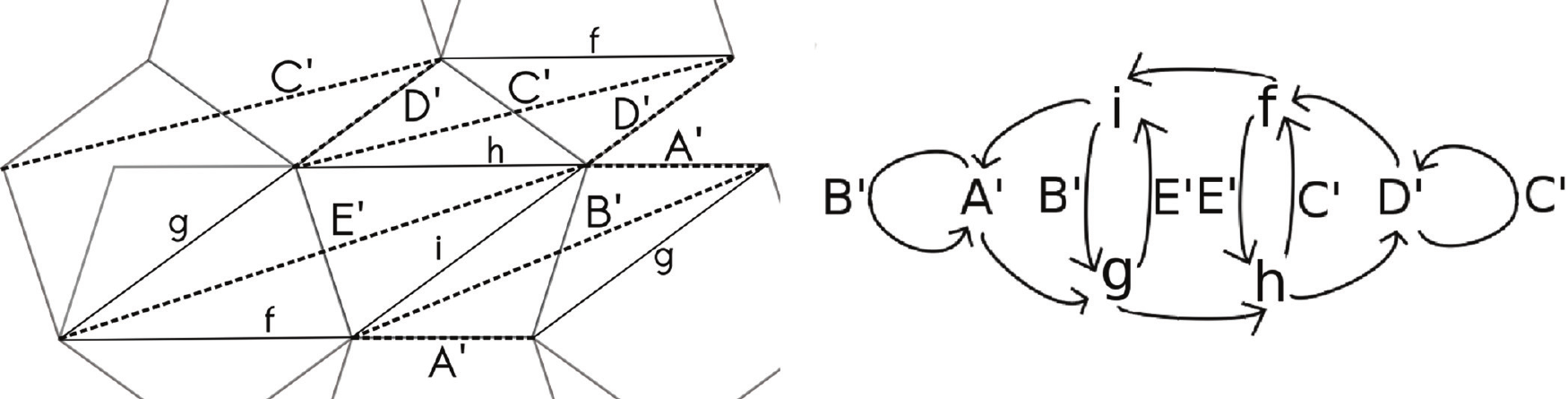}
\begin{quote}\caption{The pentagons with only auxiliary and primed edges, and the "primed" transition diagram  \label{adjacency}} \end{quote}
\end{figure}

A path on the ``unlabeled" diagram uniquely determines a path on the ``primed" diagram, since we haven't changed the diagram at all, only the labels. In our sequence, between each $i$ and $g$ we insert a $B'$; between each $h$ and $f$ we insert a $C'$, and so on, so our sequence $ghfi$ becomes $ghC'fiB'$. (It helps to write the sequence twice, so that it is clear that the $i$ at the end is matched with the $g$ at the beginning and thus needs a $B'$ between.)

To find the derived sequence, simply eliminate all of the auxiliary letters $f, g, h, i$ from the sequence. Our example sequence becomes $C'B'$, and then remove the primes, resulting in $CB$. Since the sequence is periodic, $CB = BC$, so this is the same as what we obtained when we used the sequence of pictures, when we found that the derived sequence was $BC$.

\vspace{10pt} \noindent $\mathbf{STEP \ 6}$: To complete the proof, we must see that the only letters left are the sandwiched letters. To do this, look at the dual diagram and the primed diagram together (Figure \ref{dualprimes}). If we trace out a path on the dual diagram and write down the corresponding sequence, and then trace out the same path on the primed diagram, it is clear that that the letters that are kept are exactly those that are sandwiched.
\end{proof}

\begin{figure}[!h]
\centering
\includegraphics[width=400pt]{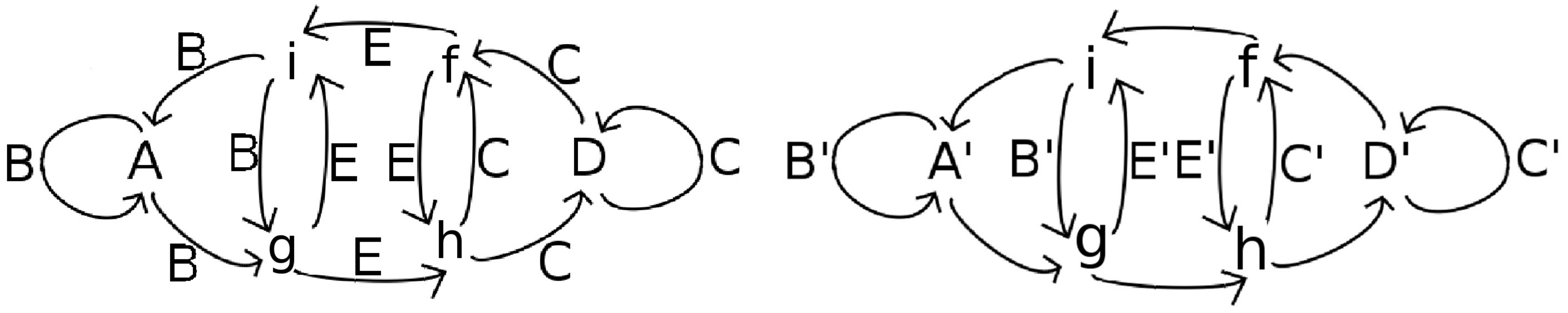}
\begin{quote}\caption{The letters we keep from the dual diagram to the primed diagram are the sandwiched letters\label{dualprimes}} \end{quote}
\vspace{-2.0em}
\end{figure}
\begin{remark}
This is exactly the same diagram as \cite{SU} found for the octagon (see \cite{SU}, Figure 15), with different labels. For double odd-gons with $n>5$, we get the same diagram except extended horizontally, as in the next section. 
\end{remark}

\subsection{Going up one step, to the heptagon \label{heptagonsection}}

In Chapter \ref{alln}, we prove this result for all double regular $n$-gons with odd $n$. Before explaining the general proof, we give another concrete example; it is instructive to examine what happens when we go up just one step, to the heptagon.

\begin{example}
For $n=7$, we can draw out the two heptagons, with primed edges and auxiliary edges, labeling the edges $A$, $B$, $\ldots$, $G$. When we draw out the transition diagram, it is similar to that of the pentagon, except that there are four sets of vertical arrows in the center of the diagram, rather than only two (Figure \ref{heptagon}). Once again, when we pass from the regular letters to the primed letters, it turns out that we eliminate the labels on the top and bottom of the transition diagram, and the labels in the middle row of the diagram become "primed." 
\end{example}

\begin{figure}[!h]
\centering
\includegraphics[width=420pt]{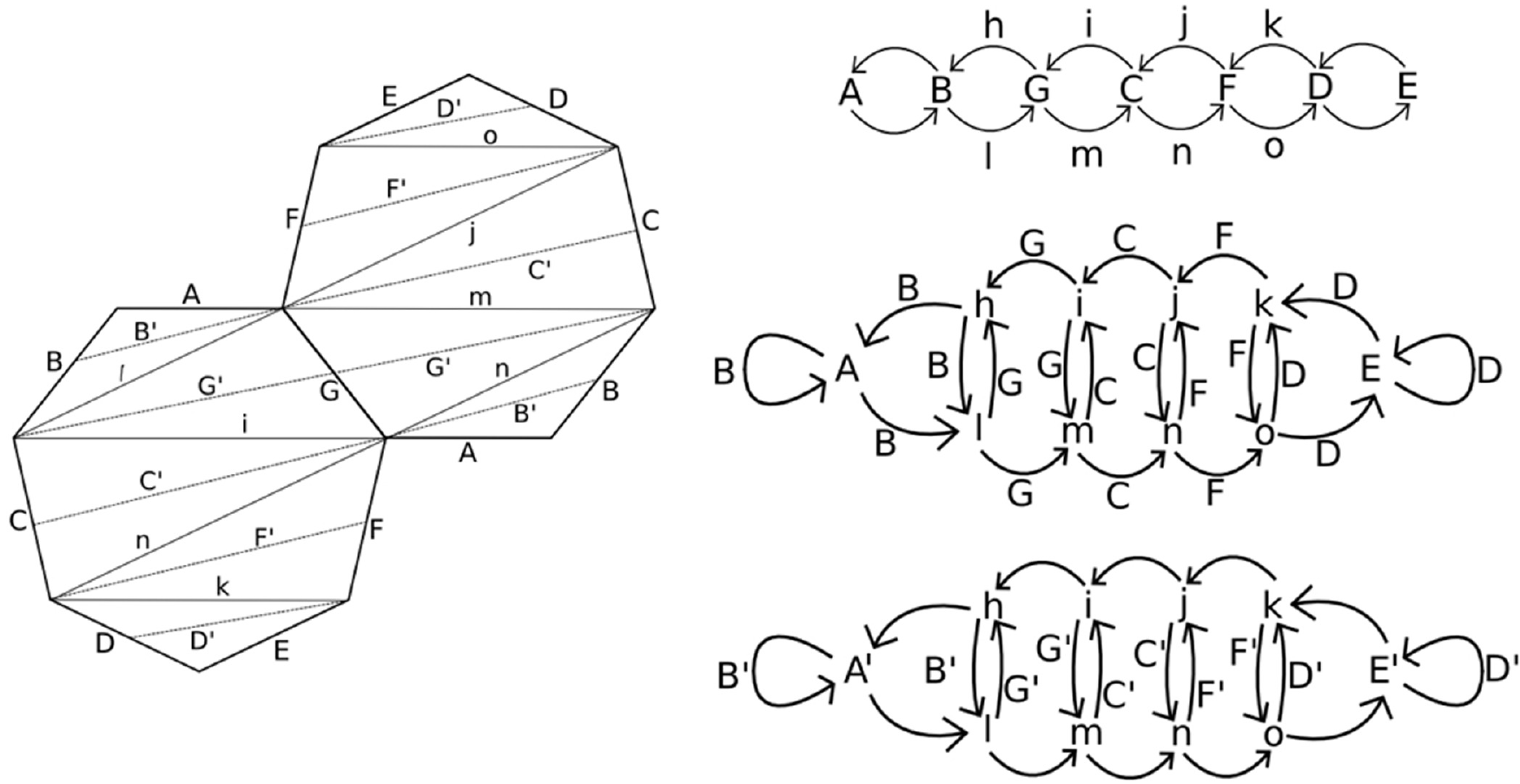}
\begin{quote}\caption{The double heptagon with auxiliary and primed diagonals drawn in, and the three transition diagrams. This is analogous to Figures \ref{arrowsletters}, \ref{adjacency} and \ref{dualprimes} for the pentagon \label{heptagon}} \end{quote}
\vspace{-2.0em}
\end{figure}

The pattern that we have observed in the cases of the double pentagon and heptagon also holds for all double odd-gons:

\begin{lemma} \label{sandwichproperty}
Whenever the labels on the top and bottom of the transition diagram (on the left- and right-facing arrows) are eliminated, and those on the middle row are kept, the effect on a cutting sequence $s$ is $ksl(s)$.
\end{lemma}

\begin{proof}
(By induction.) This can be easily checked for the base case $n=5$, the pair of transition diagrams in Figure \ref{dualprimes}, which has only one central ``square" (with corners $f$, $g$, $h$, $i$). 

Going from $n$ to $n+1$, the effect on the transition diagram is to add two additional central squares, so the problem reduces to showing that this property holds on the central squares. Since the sandwiched property is only applied to sequence fragments of length $3$, the problem reduces to showing this on a block of three consecutive central squares. This is also easily checked: Using the labels on the three central squares in Figure \ref{heptagon}, we can confirm that the middle letter is not kept for the sequence fragments $BGC$, $GCF$, $CFD$, $DFC$, $FCG$ and $CGB$, and that it is kept for $GBG$, $CGC$, $FCF$, $FDF$, $CFC$ and $GCG$, confirming the sandwiching rule. 

We checked this for the particular edge labels in the pentagon and heptagon transition diagrams, but clearly the result is the same for any congruent diagram, no matter what edge labels the arrows are given.
\end{proof}

\section{The Reassembly Theorem \label{reassembly}}

We will show that when a polygon, or pair of polygons, is sheared under an application of the Veech element, the result is exactly what we described in Chapter \ref{vertexguide}, and in particular that the images of the polygon's vertices are as we constructed them in Definition \ref{vgg} of the vertex generator guide. 

\begin{theorem} \label{theyrethesame} $\mathbf{(Reassembly)}$
The effect of applying the shearing Veech element to a double odd-gon is to perform a single twist on each of its cylinders, and the effect on the polygons in the plane is to horizontally translate their vertices to corresponding locations on the vertex generator guide.
\end{theorem}

\begin{remark}
A similar result is true for regular polygons with an even number of edges, in which case a single polygon is typically used, rather than a double polygon. In the case of a single odd-gon, some cylinders are twisted once and some twice. The vertex generator guide for even-gons is similar to that of odd-gons; see Remark \ref{whatabouteven}.
\end{remark}

\noindent $\mathbf{Outline \  of  \ Proof.}$
First, we will find the images of the polygons' vertices under transformation by the Veech element (Proposition \ref{veechelt}). Then, we will construct the vertex generator guide for a generalized regular odd-gon and find the coodinates of the sheared vertices (Proposition \ref{geometric}). Finally we will show that the two calculations agree: the sheared polygons' vertices from the geometric construction of the vertex generator guide agree with the algebraic result from the Veech element calculation.

\begin{definition}
Denote the exterior angle of a regular $n$-gon by $\alpha$, so $\alpha = 2\pi/n$ (and $\alpha/2 = \pi/n$). We will assume that in any particular situation, it will be clear which regular polygon we are using, so we can simply call the angle $\alpha$ without reference to $n$. 
\end{definition}

\begin{definition}
Consider the polygon with a horizontal edge at the bottom (the ``upper $n$-gon"). The left endpoint of the horizontal edge is the ``$0$th left-side point", and its right endpoint is the ``0th right-side point." Moving up along the left side of the polygon, the vertices are the 1st, 2nd, etc. left-side points of the upper $n$-gon, and similarly for the points on the right side. The top vertex is equivalently either the $(n-1)/2$th left-side or right-side point.

For the polygon with a horizontal edge at the top (the ``lower $n$-gon"), the left endpoint of the horizontal edge is the ``0th left-side point," and moving down along the left side, the vertices are the 1st, 2nd, etc. left-side points of the lower $n$-gon. The right endpoint of the horizontal edge is the ``0th right-side point," and moving down along the right side, the vertices are the 1st, 2nd, etc. right-side points of the lower $n$-gon. The bottom vertex is equivalently either the $(n-1)/2$th left-side or right-side point.
\end{definition}

\begin{proposition} \label{shearngons}
The coordinates of $k$th right-side point of the upper $n$-gon are 
\[
(1+\sum_{i=1}^k \cos i\alpha, \sum_{i=1}^k \sin i\alpha),
\]
and the $x$-coordinate of its image under $M_n$ is 
\begin{eqnarray}
2n+1 + \sum_{i=1}^n (4(k-i)+3)\cos i\alpha. \label{Rn1}
\end{eqnarray}
The coordinates of the $k$th left-side point of the upper $n$-gon are 
\[(-\sum_{i=1}^k \cos i\alpha, \sum_{i=1}^k \sin i\alpha),
\]
and the $x$-coordinate of its image under $M_n$ is 
\begin{eqnarray}
2n + \sum_{i=1}^k (4(k-i)+1)\cos i\alpha. \label{Ln1}
\end{eqnarray}
The coordinates of the $k$th right-side point of the lower $n$-gon are
\[
(-\cos\alpha+\sum_{i=1}^k \cos i\alpha, \sin\alpha - \sum_{i=1}^k \sin k\alpha),
\]
and the $x$-coordinate of its image under $M_n$ is
\begin{eqnarray}
2-2k+\cos \alpha - \sum_{i=1}^k (4(k-i)+1)\cos i\alpha. \label{U1}
\end{eqnarray}
The coordinates of the $k$th left-side point of the lower $n$-gon are 
\[
(-\cos\alpha-\sum_{i=1}^k \cos i\alpha, \sin\alpha - \sum_{i=1}^k \sin k\alpha),
\]
and the $x$-coordinate of its image under $M_n$ is
\begin{eqnarray}
1-2k+\cos\alpha - \sum_{i=1}^k (4(k-i)+3)\cos i\alpha. \label{U2}
\end{eqnarray}
\end{proposition}

We do not give the $y$-coordinate of the image, because it is preserved by $M_n$.

\begin{proof}
We obtain the original coordinates of the points of the $n$-gons by inspecting a diagram like Figure \ref{11gon}. We will work through the calculation for image of the $k$th right-side point of the upper $n$-gon; the other three cases are similar. 

We can find the $x$-coordinates of the upper $n$-gon's vertices by drawing in the exterior angles $\alpha$, $2\alpha$, $3\alpha$, etc. and then using cosine and sine to find the lengths. The $x$-coordinates are found using cosine, and those lengths are labeled in Figure \ref{11gon}; the $y$-coordinates are found similarly, with sine. An analogous picture is used for the lower $n$-gon.

\begin{figure}[!h]
\centering
\includegraphics[width=200pt]{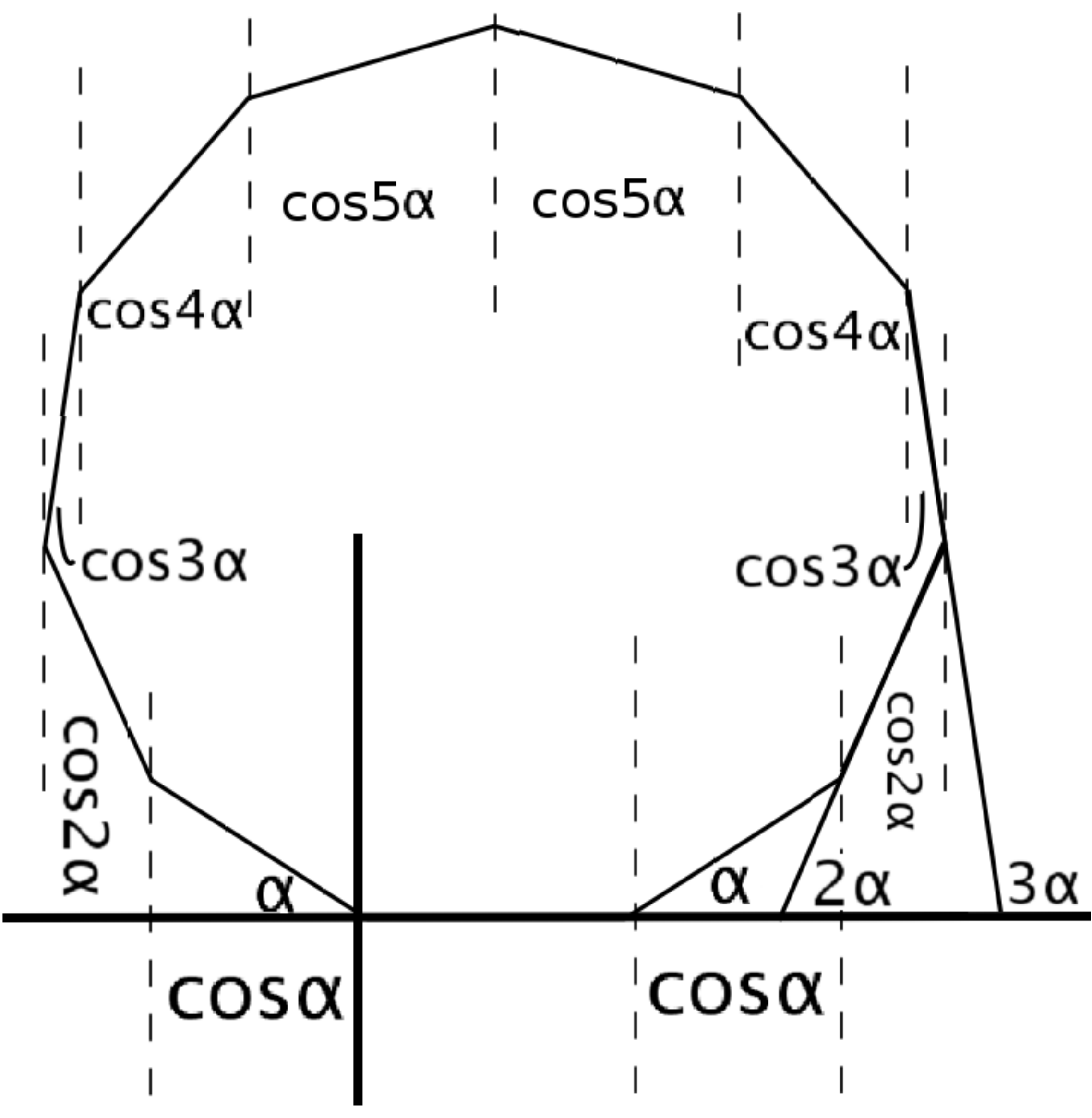}
\begin{quote}\caption{Finding the $x$-coordinates of the upper $n$-gon's vertices  \label{11gon}} \end{quote}
\end{figure}

A right-side point on the top $n$-gon (i.e., the first of the four cases in the Proposition) is of the form $(1+\cos\alpha + \ldots + \cos k\alpha,\sin\alpha+\ldots+\sin k\alpha)$. To apply the shearing Veech element $M_n$ to this point, we multiply (see Proposition \ref{veechelt}) 
\begin{eqnarray*}
\begin{pmatrix}
  1&2\cot(\alpha/2)
  \\ 0&1
\end{pmatrix}
\begin{pmatrix}
 1+\sum_{i=1}^k \cos i\alpha \\
\sum_{i=1}^k \sin i\alpha
\end{pmatrix}.
\end{eqnarray*}
Note that we have used the shearing Veech element $M_n$ here, rather than the generator $\bigl( \begin{smallmatrix} -1&0\\ 0&1 \end{smallmatrix} \bigr) M_n$ of the Veech group, so that we have only a shear, without the horizontal flip. Preserving the orientation for our calculation simplifies our understanding of what is going on. Once we understand the shear, we will flip the picture horizontally, applying the Veech group generator; see Remark \ref{flip} at the end of the chapter.

The $y$-coordinate is unchanged, and the $x$-coordinate is 
\begin{eqnarray*}
1+\sum_{i=1}^k \cos i\alpha &+& 2\cot \alpha/2 \sum_{i=1}^k \sin i\alpha \\
= 1+\sum_{i=1}^k \cos i\alpha &+& 2 \sum_{i=1}^k \cot \alpha/2 \sin i\alpha.
\end{eqnarray*}
Applying Lemma \ref{identitysum} yields
\begin{eqnarray*}
&=& 1+\sum_{i=1}^k \cos i\alpha + 2 \big ( k+ \sum_{i=1}^k (2(k-i)+1)\cos i\alpha \big ) \\
&=& 1+\sum_{i=1}^k \cos i\alpha + 2k + \sum_{i=1}^n (4(k-i)+2) \cos i\alpha \\
&=& 2k+1 + \sum_{i=1}^n (4(k-i)+3) \cos i\alpha, 
\end{eqnarray*}
as desired.

The calculations for the other three types of points are the same.
\end{proof}

When we construct the vertex generator guide as described in Chapter \ref{vertexguide}, we get a picture as in Figure \ref{11gongrid}. (Note that this is only for the upper $n$-gon; we can draw a similar picture for the lower $n$-gon.)  The original locations of the $n$-gon's vertices are in gray on the left side of the diagram, and their sheared images on the vertex generator guide are in black.\label{11gongrid} In the figure, the numbers indicate horizontal distances: the label $k$ represents a horizontal distance of $\cos k\alpha$ (see Figure \ref{11gon}). Additionally, each horizontal edge has length $1$.

\begin{figure}[!h]
\centering
\includegraphics[width=430pt]{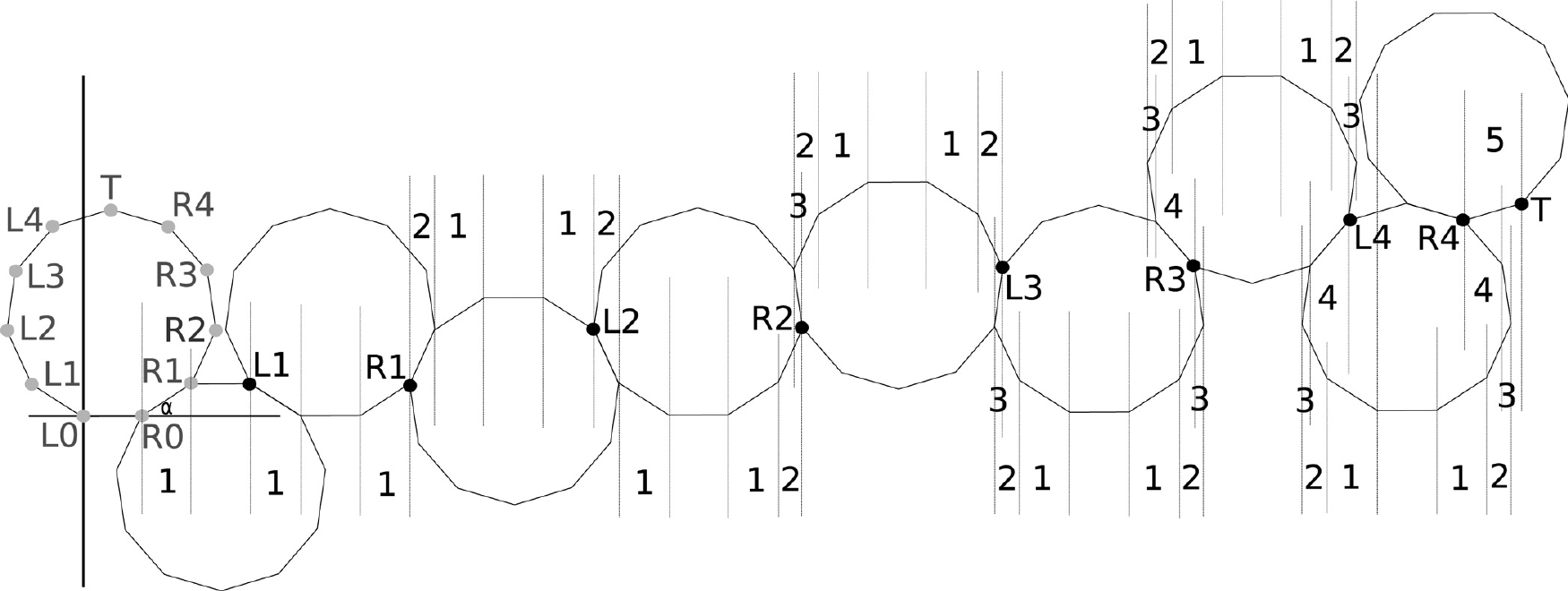}
\begin{quote}\caption{A diagram for finding the $x$-coordinates of the points in the vertex generator guide } \end{quote}
\end{figure}

The $y$-coordinates are unchanged, and we can find the $x$-coordinates of the vertices by finding sums of the edge vectors. We do this in the following proposition:

\begin{proposition} \label{geometric}
The $x$-coordinates of the vertices of the sheared upper $n$-gon in the vertex generator guide are given by:

\begin{eqnarray}
R_k &=& 2k+1 + \sum_{i=1}^k (4(k-i)+3)\cos i\alpha, \label{Rn2} \\
L_k &=& 2k + \sum_{i=1}^k (4(k-i)+1)\cos i\alpha. \label{Ln2}
\end{eqnarray}
\end{proposition}

\begin{proof}
We find the coordinates by ``snaking" along the edges of the $n$-gons, as in Figure \ref{11gongrid}, from one point of the vertex generator guide to the next, always choosing the path around the $n$-gon that includes the horizontal edge, so that our angles have the smallest possible multiple of $\alpha$ (Figure \ref{snake}). When we do this, we obtain the above expressions for $R_n$ and $L_n$, as follows: 

Consider the leading term $2k+1$ of $R_k$, which measures the number of horizontal edges we use to reach $R_k$ from $L_0$. It is clear that for $R_0$, this number is 1. We have to move along two polygons to get from one $R_k$ to the next, and we always choose the route that includes a horizontal edge, so when $k$ increases by 1, the number of horizontal edges increases by 2. Thus, the leading term should be $2k+1$. 

The other terms can be found similarly, looking at the first instance in which $\cos i\alpha$ appears for a given value of $i$, and then proceeding by induction as above.
\end{proof}

\begin{figure}[!h]
\centering
\includegraphics[width=430pt]{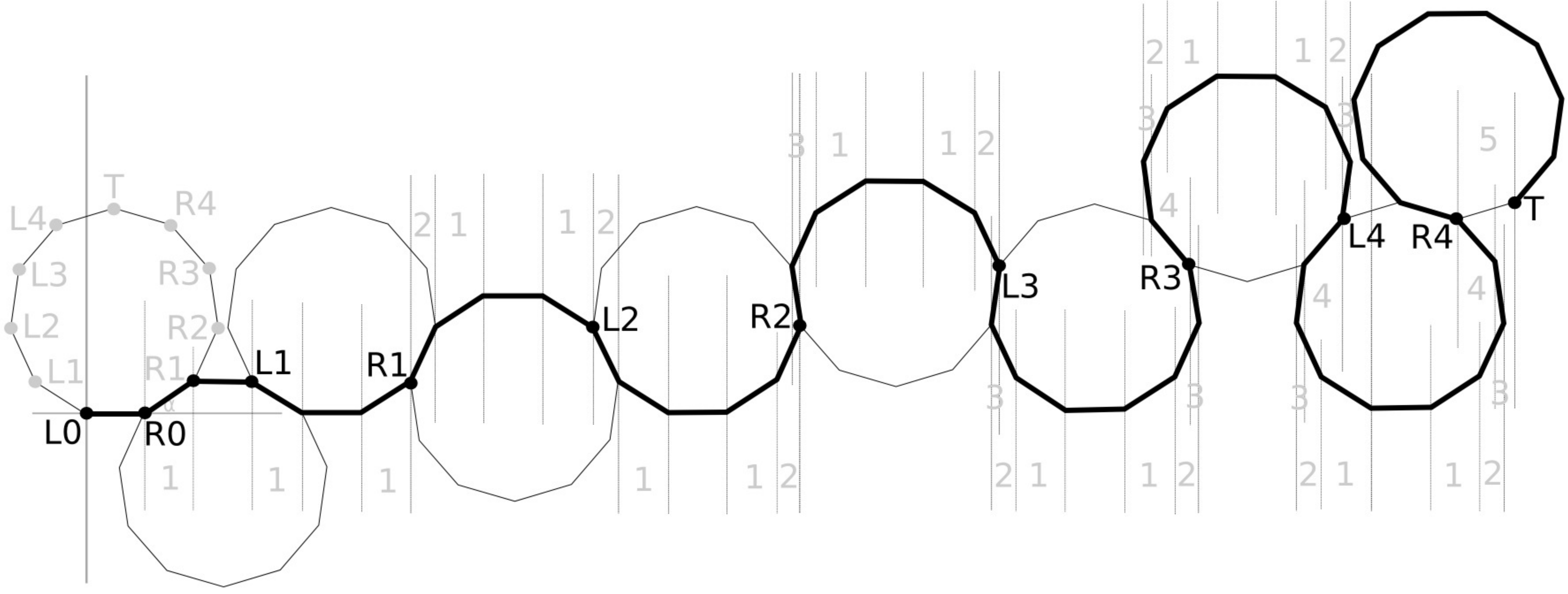}
\begin{quote}\caption{We find the coordinates of the vertex generator guide points by ``snaking" around the figure\label{snake}} \end{quote}
\end{figure}

Note that when we have found the coordinates of $T$, we are done; for example, in Figure \ref{snake} we have found \emph{all} the vertices for the upper $n$-gon.

\begin{remark}
We only give expressions for $L_n$ and $R_n$, never for the image of the top vertex $T$. This is because for a double $n$-gon, if $m=(n-1)/2$, $T=R_m=L_m$, so its coordinates can be obtained from either of the above equations.
\end{remark}

\begin{proposition} \label{lowerngon}
The $x$-coordinates of the vertices along the right side of the sheared lower $n$-gon in the vertex generator guide are given by
\begin{eqnarray}
R_n &=& 2-2k+\cos \alpha - \sum_{i=1}^k (4(k-i)+1)\cos i\alpha. \label{Rn3}
\end{eqnarray}
and the $x$-coordinates of the vertices along the left side of the sheared lower $n$-gon in the vertex generator guide are given by
\begin{eqnarray}
L_n &+& 1-2k+\cos\alpha - \sum_{i=1}^k (4(k-i)+3)\cos i\alpha. \label{Ln3}
\end{eqnarray}
\end{proposition}

\begin{proof}
The figure, table and general expressions for the lower $n$-gon are found in the same way as for the upper $n$-gon, so we omit this calculation.
\end{proof}

We are now ready to prove the Reassembly Theorem, given at the beginning of this chapter.

\begin{proof}
(Proof of Theorem \ref{theyrethesame}) The geometric description of shearing in Propositions \ref{geometric} and \ref{lowerngon} agree with the algebraic description in Proposition \ref{shearngons} of the image points under the Veech group: The geometric description of the right-side points in the upper $n$-gon in Equation (\ref{Rn1}) agrees with the geometric description in Equation (\ref{Rn2}), and the left-side algebraic equation (\ref{Ln1}) is also identical to the geometric description in equation (\ref{Ln2}). Similar geometric calculations for the lower $n$-gon show that the geometrically-derived expressions in Equations (\ref{Rn3}) and (\ref{Ln3}) agree with the algebraically-derived coordinates in equations (\ref{U1}) and (\ref{U2}). Since the coordinates are the same, the geometric construction of the vertex generator guide agrees with the algebraic result of applying the Veech element.
\end{proof}

\begin{remark} \label{flip}
We have shown that the shearing Veech element $M_n$ acts on the vertices of a standard double $n$-gon by sliding them to corresponding locations in the vertex generator guide. When we apply the reflecting and shearing Veech element $\bigl( \begin{smallmatrix} -1&0\\ 0&1 \end{smallmatrix} \bigr) M_n$, it simply reverses the orientation, exchanging vertex $L_k$ with $R_k$ for each $k$. Since the orientation reversal applies to the Veech element's geometric and algebraic effects in the same way, the two descriptions still agree.
\end{remark}

\section{Derivation for all odd $n$-gons \label{alln}}

Finally, we will show that derivation amounts to applying the rule $ksl(s)$ for a sequence $s$ on any double regular odd $n$-gon.

In Proposition \ref{shearngons}, we placed a $1$ in the upper-left corner of the Veech element (shearing matrix) instead of a $-1$. The result of placing a $-1$ there instead is a horizontal flip, as we discussed in our detailed exploration of the double pentagon in Chapter \ref{deriveBECE}. This simply exchanges the labels of the edges. For example, in the double pentagon, edges $B$ and $E$ are exchanged, and edges $C$ and $D$ are exchanged, and side $A$ is fixed (but reversed in its orientation). This transformation simplifies the procedure significantly. Without it, derivation would be more complicated than the simple rule that defines $ksl(s)$.

\begin{definitions} \label{definitions}
The $n$-gon on the left is the \emph{lower $n$-gon} ($P_L$), and the one on the right is the \emph{upper $n$-gon} ($P_U$). The horizontal edge is $S_1$, and the edges are numbered counter-clockwise up to $S_n$. The image of an edge $S_k$ after the Veech element is applied is $S_k$', and these are called the \emph{primed edges}. Auxiliary edges in the lower $n$-gon are named $L_k$, and those in the upper $n$-gon are named $U_k$. Finally, for ease of notation in labeling the edges of the polygon farthest from edges $S_1$ and $S_n$, we will let $m = (n+1)/2$. Note that this is always an integer, since $n$ is odd. These notations are used in Figure \ref{double17gon}.

``Deriving" a cutting sequence $s$ means removing all letters except those that are sandwiched, i.e. applying the function $ksl(s)$.
\end{definitions}

\begin{lemma} \label{midpoint}
Edge $S_k$' cuts through edge $S_k$, except for horizontal edges, in which case $S_k$' = $S_k$.
\end{lemma}
\begin{proof}
The effect of the Veech element on a horizontal edge $S_k$ is horizontal translation to a congruent edge $S_k$'. We cut up the pieces of the sheared polygons and reassemble them back into the standard double polygon, so translation has no effect, and $S_k$' = $S_k$.

If we perform a cylinder decomposition of a double polygon, each cylinder (which is a parallelogram) is made of two smaller polygons (trapezoids or triangles), one from $P_U$ and one from $P_L$, glued along an edge that is inside the parallelogram. Every non-horizontal edge $S_k$ appears as this interior gluing edge of some cylinder decomposition parallelogram (Figure \ref{skintskprime}). We will show that $S_k$' cuts through edge $S_k$ in two steps:

\begin{figure}[!h]
\centering
\includegraphics[width=280pt]{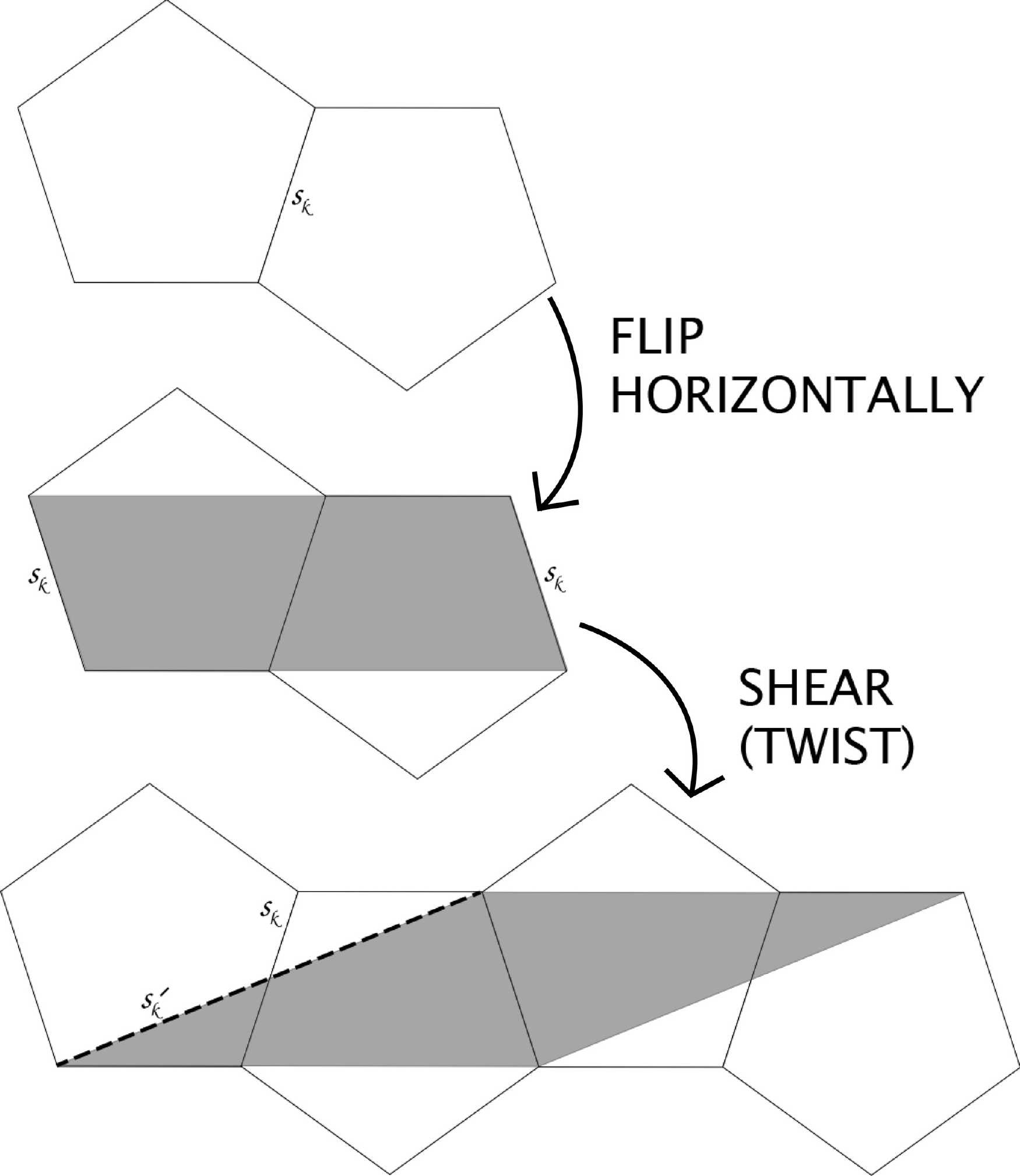}
\begin{quote}\caption{Showing that edge $S_k$ intersects edge $S_k$' \label{skintskprime}} \end{quote}
\end{figure}

\vspace{10pt} \noindent $\mathbf{STEP \ 1:}$ We apply a horizontal flip. This exchanges the label of the interior gluing edge with the labels of the left and right sides of the parallelogram. After this flip, $S_k$ is the left and right edges of the parallelogram (the second diagram in Figure \ref{skintskprime}).

\vspace{10pt} \noindent $\mathbf{STEP \ 2:}$ We apply the shearing Veech element $M_n$. The effect of the Veech element is to perform a single twist in each cylinder. Consider the copy of $S_k$ appearing on the left side of the parallelogram. The shear fixes the lowest point of this edge and takes the highest point of the edge to the copy of $S_k$ on the right side of the cylinder. In other words, the sheared edge $S_k$' is a diagonal across the parallelogram from the lower-left corner to the upper-right corner (see the third diagram in Figure \ref{skintskprime}). Since the interior edge, which was the initial position of $S_k$, touches the top and bottom horizontal edges of the parallelogram, it must intersect the diagonal $S_k$'.
\end{proof}

\begin{lemma} \label{nminus3}
We need $n-3$ auxiliary edges in an $n$-gon.
\end{lemma}
\begin{proof}
Adding auxiliary edges is a triangulation, which requires $n-3$ triangles. Alternatively, there are $2$ auxiliary edges in a pentagon, and each time we add two edges, we increase the number of auxiliary edges by $2$.
\end{proof}

Lemmas \ref{midpoint} and \ref{nminus3} allow us to draw a picture of our double $n$-gon, with primed edges and auxiliary edges, which we do in Figure \ref{double17gon}. For the pentagon, we drew this in two steps: First, we drew the double pentagon with original edges and auxiliary edges, and then we drew a double pentagon with auxiliary edges and primed edges (Figures \ref{dualletters} and \ref{adjacency}). Now, we have combined all the information into one picture, as we did with the heptagon in Section \ref{heptagonsection}.

\begin{figure}[!ht]
\centering
\includegraphics[width=390pt]{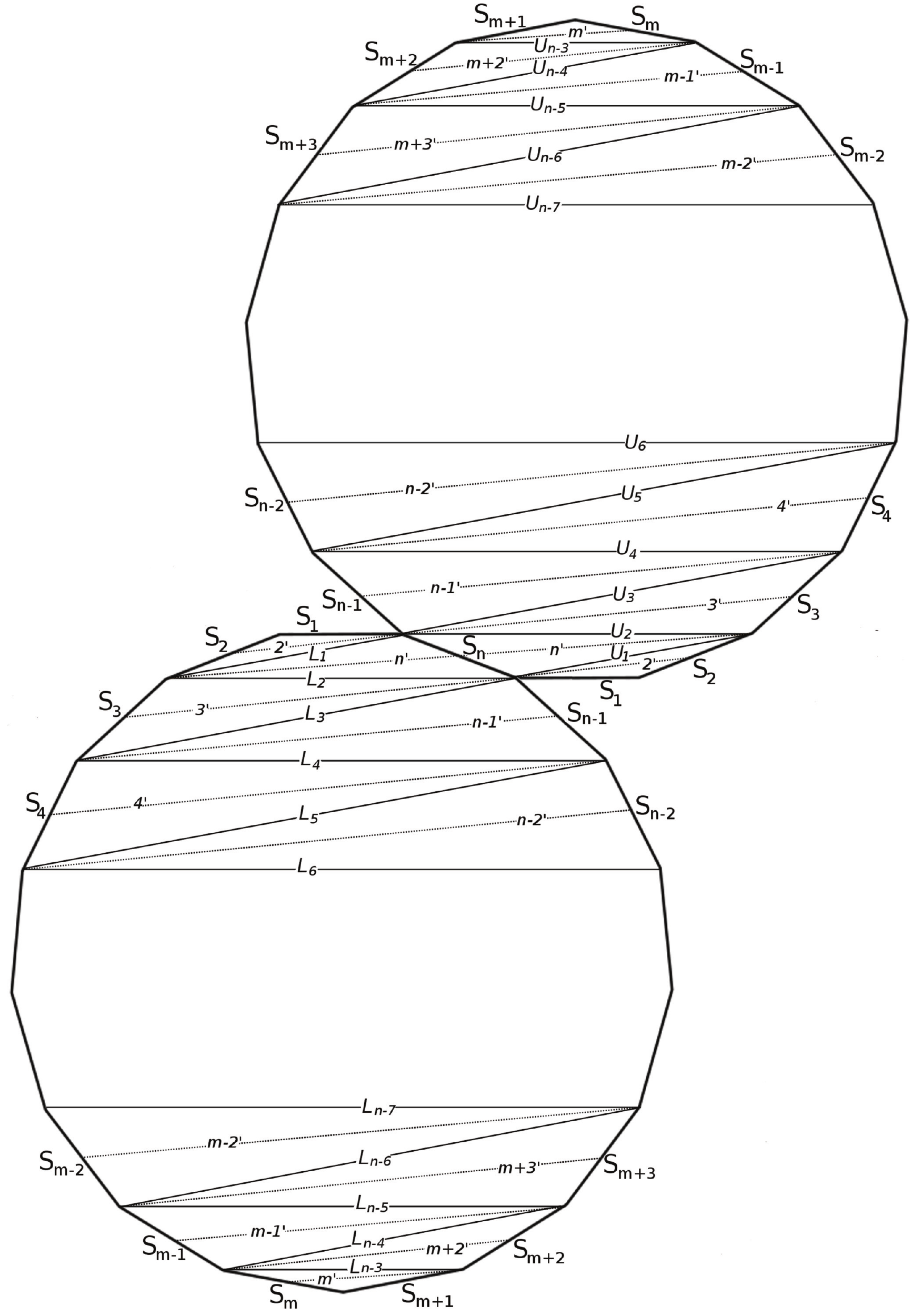}
\begin{quote}\caption{A general double $n$-gon for odd $n$, with both primed and auxiliary edges included\label{double17gon}} \end{quote}
\end{figure}

Now we can prove that the effect of applying the Veech shear is derivation of the cutting sequence for all double odd-gons. Again, the proof will consist of $6$ steps, exactly analogous to the $6$ steps for the double pentagon in Section \ref{sequences}.

\begin{theorem} \label{koslalln}
Given a bi-infinite trajectory on a double regular $n$-gon, with $n$ odd, and its corresponding cutting sequence, the cutting sequence of the trajectory obtained by applying the Veech element to the double $n$-gon is the derivation of the original cutting sequence.
\end{theorem}

\begin{proof}
We will show that all of the labels on the top and bottom horizontal arrows of the transition diagram are eliminated, and the edges $S_k$ along the middle row on the vertical arrows are replaced with $S_k$'; this is sufficient to prove that only sandwiched letters are kept, as shown in Lemma \ref{sandwichproperty}. 

\vspace{10pt} \noindent $\mathbf{STEP \ 1:}$ First, we need to construct a transition diagram. Instead of labeling our edges $A,B$, etc. we use $S_1,\ldots,S_n$. We assume that all trajectories go up, and that $\theta \in [0,\pi/n)$. To obtain the diagram, we glue the $n$-gons in Figure \ref{double17gon} together along various identified edges, to see which auxiliary edges must be crossed to go from one edge to another. All of these gluings follow a similar pattern, so we can create a parallelogram ``building block" of the interactions between the various edges, which is shown at the top of Figure \ref{buildingblock}.  

\begin{figure}[!h]
\centering
\includegraphics[width=430pt]{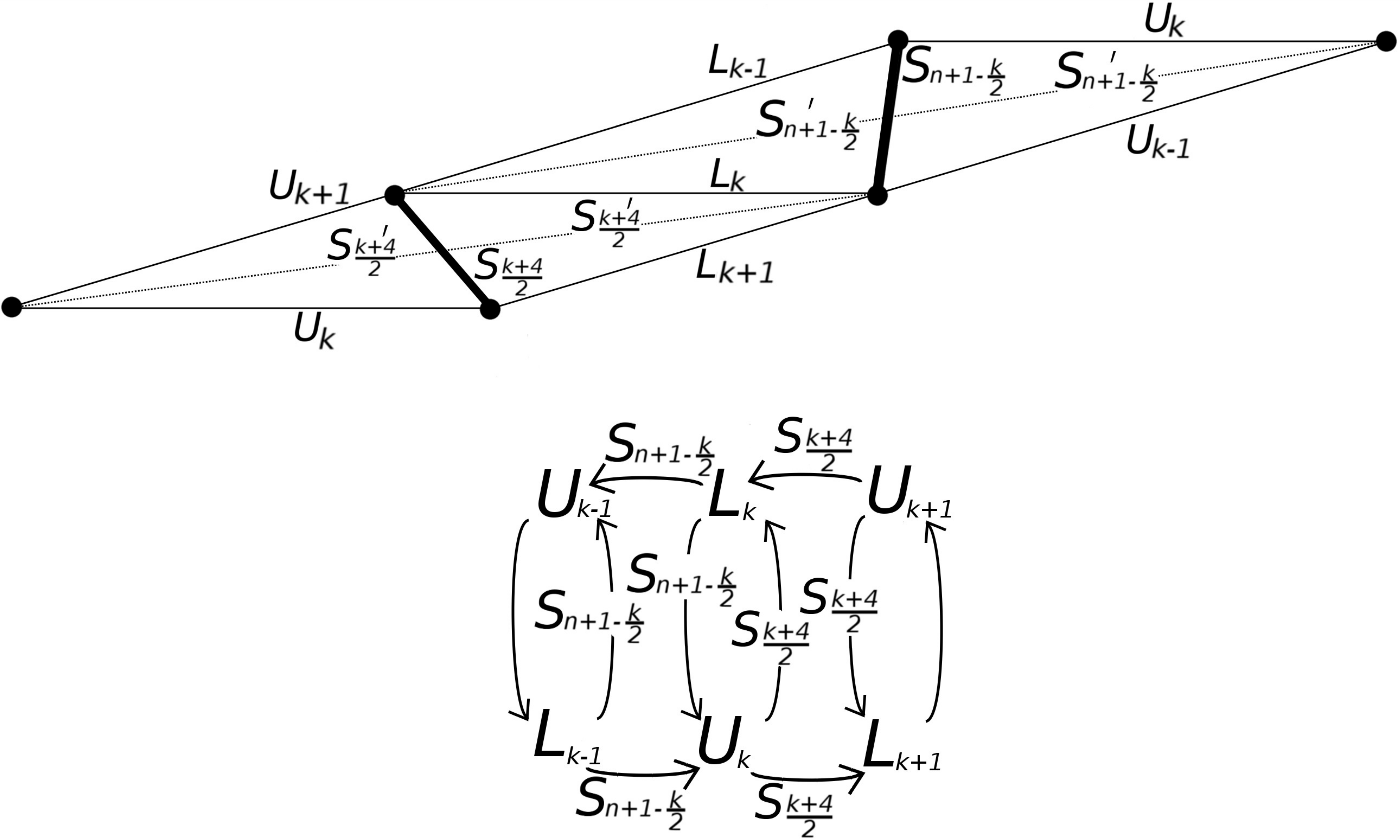}
\begin{quote}\caption{A ``building block" showing the relative positions of various edges. Edges of the original polygons are bold; their vertices are highlighted with dots; auxiliary diagonals are thin lines; and primed edges are dotted lines \label{buildingblock}} \end{quote}
\end{figure}

We use the parallelogram building block to construct a transition diagram (see the top diagram in Figure \ref{transitions}, ignoring the arrow labels). A cutting sequence on the $n$-gon is a bi-infinite sequence taken from the letters ${S_1,\ldots,S_n}$, which can be seen as a path around the transition diagram.

\begin{figure}[!h]
\centering
\includegraphics[width=430pt]{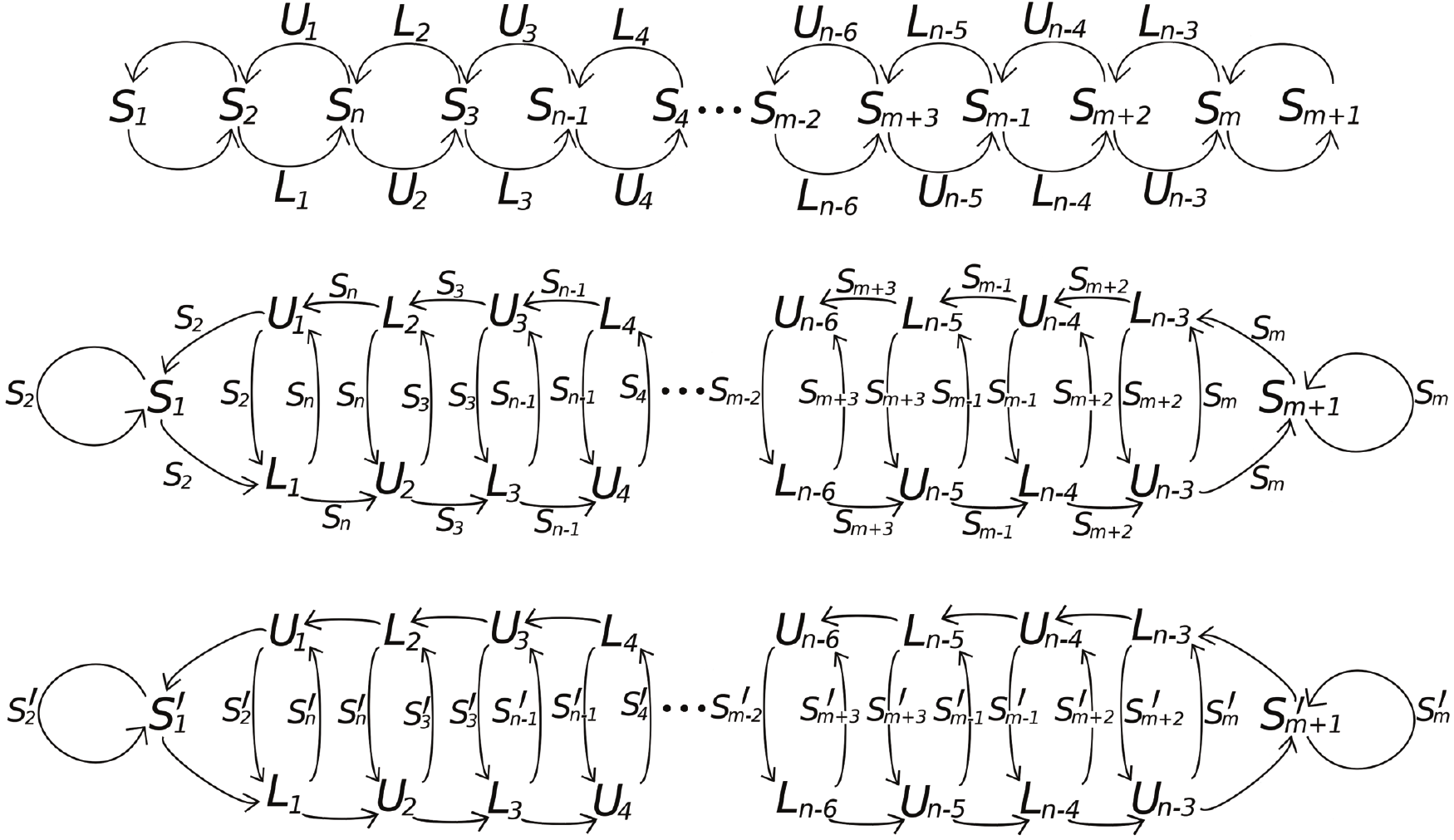}
\begin{quote}\caption{The transition diagram for a general double $n$-gon with odd $n$. (a) The basic transition diagram for original edges and auxiliary edges (b) The dual transition diagram, exchanging the roles of original and auxiliary edges (c) The transition diagram using primed edges instead of original edges \label{transitions}} \end{quote}
\end{figure}

\vspace{10pt} \noindent $\mathbf{STEP \ 2:}$ We augment the lower and upper polygons with the auxiliary edges $L_i$ and $U_i$, respectively, at angles $\theta=0$ and $\theta=\pi/n$. Then we label the arrows of the transition diagram, using the parallelogram building block to determine which auxiliary edges are crossed to get from one of the polygon's edges to another. We augment our cutting sequence with the auxiliary edge labels, so our sequence is now composed of the letters $\{S_i,U_i,L_i\}$.

The extreme ends of the transition diagram are easy to check: The left side of the transition diagram corresponds to the middle part of Figure \ref{double17gon}, where the polygons are shown as being glued together. The right side of the transition diagram corresponds to the top and bottom of Figure \ref{double17gon}. The middle of the diagram corresponds to the intermediary edges on both sides, and can be built from multiple copies of the transition diagram building block at the bottom of Figure \ref{buildingblock}. 

To determine the labels of this transition diagram building block, we can use the top picture in Figure \ref{buildingblock}. The quadrilateral in the middle is a piece of the lower part of a lower $n$-gon. On both sides, we have glued triangles cut from the upper part of an upper $n$-gon. Using these, we can determine which edges must be crossed to go from each $L_i$ to each $U_j$, and fill in the middle of the transition diagram.

\vspace{10pt} \noindent $\mathbf{STEP \ 3:}$ We take "the dual" of the transition diagram, as we did in the pentagon to go from Figure \ref{arrowsletters} to Figure \ref{dualletters}, obtaining the middle picture in Figure \ref{transitions}. A path on the top diagram uniquely determines a path on the dual diagram, and vice-versa. We could also obtain this diagram from the building block in Figure \ref{buildingblock}, by looking at which original edges ($S_i$) and primed edges ($S_i'$) are crossed to go from one auxiliary edge ($U_i$ or $L_i$) to another. This does not change our sequence; it is still composed of the letters $\{S_i,U_i,L_i\}$.

\vspace{10pt} \noindent $\mathbf{STEP \ 4:}$ Our goal is to end up with just the primed edges that our trajectory crosses. We use the parallelogram building block in Figure \ref{buildingblock} again, this time ignoring the original edges ($S_i$) and looking at which primed edges ($S_i'$) are crossed to go from one auxiliary edge ($U_i$ or $L_i$) to another. This gives us the bottom diagram in Figure \ref{transitions}. Our cutting sequence is now composed of the letters  $\{S_i',U_i,L_i\}$.

\vspace{10pt} \noindent $\mathbf{STEP \ 5:}$ In our cutting sequence, we eliminate the auxiliary letters $\{U_i,L_i\}$, so our sequence is entirely composed of primed letters $\{S_i'\}$. To obtain the derived sequence, simply eliminate the "primes" to end up with the derived sequence in the letters $\{S_i\}$.

\vspace{10pt} \noindent $\mathbf{STEP \ 6:}$ We look at the middle and bottom diagrams in Figure \ref{transitions}. To get from the middle to the bottom, we eliminate the arrows on the top left-facing arrows and the bottom right-facing arrows, and we keep all the middle arrows on the up- and down-facing arrows, but make them "primed." As shown in Lemma \ref{sandwichproperty}, this is equivalent to keeping only sandwiched letters, so this completes the proof.
\end{proof}

{\footnotesize \bibliography{thesis}}

\end{document}